\documentclass{jloganal}    

\usepackage{amsfonts,amsbsy,xspace, latexsym, amscd, graphicx, fancyhdr,verbatim,enumitem}
\title[Linear Algebra and Solvability of PDE]{Infinite-Dimensional Linear Algebra and Solvability of Partial Differential Equations}

\author{Todor D. Todorov}
\givenname{Todor}
\surname{Todorov}
\address{Professor Emeritus,                
	California Polytechnic State University\\
	San Luis Obispo, California 93407, USA.}
\email{ttodorov@calpoly.edu}
\urladdr{https://math.calpoly.edu/todor-todorov}

\dedicatory {The article is dedicated to Professor Michael Oberguggenberger on occasion of his 65th Birthday}
\arxivreference{}
\arxivpassword{}

\volumenumber{}
\issuenumber{}
\publicationyear{}
\papernumber{}
\startpage{}
\endpage{}
\doi{}
\MR{}
\Zbl{}
\received{}
\revised{}
\accepted{}
\published{}
\publishedonline{}
\proposed{}
\seconded{}
\corresponding{}
\editor{}
\version{}

\newcommand{\bra}{\ensuremath{\langle}}
\newcommand{\ket}{\ensuremath{\rangle}}

\newcommand{\spm}{{\ensuremath{\rm sp}}}
\newcommand{\supp}{\ensuremath{{\rm{supp}}}}
\newcommand{\card}{\ensuremath{{\rm{card}}}}

\newcommand{\ran}{\ensuremath{{\rm{ran}}}}

\newcommand{\rest}{\!\ensuremath{\upharpoonright}\!}

\newcommand{\Diff}{\ensuremath{{\rm{Diff}}}}

\newcommand{\ifff}{\ensuremath{{\rm{\,iff\;}}}}

\newlist{A-enum}{enumerate}{1}
\newlist{T-enum}{enumerate}{2}
\newlist{L-enum}{enumerate}{2}
\newlist{C-enum}{enumerate}{2}
\newlist{P-enum}{enumerate}{2}  
\newlist{Pf-enum}{enumerate}{2} 
\newlist{D-enum}{enumerate}{2}
\newlist{Ex-enum}{enumerate}{2}
\newlist{Exs-enum}{enumerate}{2}
\newlist{E-enum}{enumerate}{2}
\newlist{R-enum}{enumerate}{2}
\setlist[A-enum,1]{label=(\alph*),format=\bfseries\emph,leftmargin=*,labelindent=.1\parindent}
\setlist[T-enum,1]{label=(\roman*),format=\bfseries\emph,leftmargin=*,labelindent=.1\parindent}
\setlist[T-enum,2]{label=(\alph*),format=\bfseries\emph,leftmargin=*,labelindent=.1\parindent}
\setlist[L-enum,1]{label=(\roman*),format=\bfseries\emph,leftmargin=*,labelindent=.1\parindent}
\setlist[L-enum,2]{label=(\alph*),format=\bfseries\emph,leftmargin=*,labelindent=.1\parindent}
\setlist[C-enum,1]{label=(\roman*),format=\bfseries\emph,leftmargin=*,labelindent=.1\parindent}
\setlist[C-enum,2]{label=(\alph*),format=\bfseries\emph,leftmargin=*,labelindent=.1\parindent}
\setlist[P-enum,1]{label=(\roman*),format=\bfseries\emph,leftmargin=*,labelindent=.1\parindent}
\setlist[P-enum,2]{label=(\alph*),format=\bfseries\emph,leftmargin=*,labelindent=.1\parindent}
\setlist[Pf-enum,1]{label=(\roman*), leftmargin=*,labelindent=.1\parindent}
\setlist[Pf-enum,2]{label=(\alph*), leftmargin=*,labelindent=.1\parindent}
\setlist[D-enum,1]{label=\textbf{\arabic*.},leftmargin=*,labelindent=.2\parindent}
\setlist[D-enum,2]{label=\textbf{(\alph*)},leftmargin=*,labelindent=.1\parindent}
\setlist[Ex-enum,1]{label=\textbf{\arabic*.},leftmargin=*,labelindent=.15\parindent}
\setlist[Exs-enum,1]{label=\textbf{\arabic*.},leftmargin=*,labelindent=.15\parindent}
\setlist[E-enum,1]{label=\textbf{\arabic*.},leftmargin=*,labelindent=.15\parindent}
\setlist[R-enum,1]{label=\textbf{\arabic*.},leftmargin=*,labelindent=.15\parindent}
\theoremstyle{definition}
\newtheorem{definition}{Definition}[section]
\newtheorem{remark}[definition]{Remark}
\newtheorem{remarks}[definition]{Remarks}
\newtheorem{example}[definition]{Example}
\newtheorem{examples}[definition]{Examples}

\theoremstyle{plain}
\newtheorem{theorem}[definition]{Theorem}

\newtheorem{lemma}[definition]{Lemma}
\newtheorem{corollary}[definition]{Corollary}

\begin{document}

	\begin{abstract}    
		We discuss linear algebra of infinite-dimensional vector spaces in terms of algebraic (Hamel) bases. As an application we prove the surjectivity of a large class of linear partial differential operators with smooth ($\mathcal C^\infty$-coefficients) coefficients, called in the article \emph{regular}, acting on the algebraic dual $\mathcal D^*(\Omega)$ of the space of test-functions $\mathcal D(\Omega)$. The surjectivity of the partial differential operators guarantees solvability of the corresponding partial differential equations within $\mathcal D^*(\Omega)$. We discuss our result in contrast to and comparison with similar results about the restrictions of the regular operators on the space of Schwartz distribution $\mathcal D^\prime(\Omega)$, where these operators are often non-surjective. 
	\end{abstract}
	
	\maketitle
	


\section{Introduction}\label{S: Introduction}

In Section~\ref{S: Infinite-Dimensional Linear Algebra}-\ref{S: Linear Maps and Operators} we present the basic results of  \emph{infinite-dimensional linear algebra}, an old branch of mathematics initiated in1905 by Georg Hamel~\cite{gHamel}, dealing with infinite-dimensional vector spaces in terms of algebraic (Hamel) bases rather than topological or orthonormal Hilbert bases. The approach is mostly algebraic. In Theorem~\ref{T: Surjective Dual} we show that a linear operator is injective if and only if its dual operator is surjective; a result well-known for finite-dimensional vector spaces but less-known for infinite-dimensional spaces. This gives rise to the Definition~\ref{D: Regular Operators} of a \emph{regular linear operator} - a surjective operator on the dual space with injective co-dual. 

Several discussions of the earlier versions of this text convinced us that 
the algebraic (Hamel) bases have gradually been falling out of popularity in the last several decades. That is why the first part of the article (Section~\ref{S: Infinite-Dimensional Linear Algebra}-\ref{S: Linear Maps and Operators}) is written in somewhat tutorial manner, with many illustrative examples (Section~\ref{S: Examples of Infinite-Dimensional Spaces}). A reader who knows Theorem~\ref{T: Surjective Dual} from the finite-dimensional linear algebra and who believes in its validity for  infinite-dimensional vector spaces might skip reading the first several sections and start directly from Section~\ref{S: Linear Functionals in D*(Omega) as Generalized Functions: The Main Result}. 

In Sections~\ref{S: Linear Functionals in D*(Omega) as Generalized Functions: The Main Result} we apply infinite-dimensional linear algebra to the particular case of the vector space $\mathcal D(\Omega)$ and its algebraic dual $\mathcal D^*(\Omega)$. Here $\Omega$ is an open set of $\R^d$ in the \emph{usual topology of} $\R^d$. Somewhere in this section we abandon the realm of algebra and start involving concepts and methods from functional analysis and the theory of partial differential operators (H\"{o}rmander~\cite{lHor-I}-\cite{lHor-IV}). In particular, Definition~\ref{D: Regular Operators} (mentioned above) - if applied to $\mathcal D(\Omega)$ - gives rise to the concept of \emph{regular operator} \emph{with $\mathcal C^\infty$-coefficients}: a surjective linear operator   $P^*(x, \partial)$ with $\mathcal C^\infty$-coefficients acting on $\mathcal D^*(\Omega)$ which has an injective co-dual (transposed) operator $P(x, \partial)$ on $\mathcal D(\Omega)$. 

For readers \emph{without background in Schwartz theory of distributions}
(Vladimirov~\cite{vVladimirov}) who are otherwise interested in the main topic of our article, we offer a characterization of the space of Schwartz distributions  $\mathcal D^\prime(\Omega)$ as a particular subspace of  $\mathcal D^*(\Omega)$ without the usual involvement of the strong topology on the space of test-functions $\mathcal D(\Omega)$ (Section~\ref{S: Schwartz Distributions within D*(Omega)}). We  shortly outline a \emph{sequential approach to distribution theory} based on our characterization (Remark~\ref{R: Sequential Approach to Distribution Theory}). Thus, the dilemma $\mathcal D^\prime(\Omega)$ vs. $\mathcal D^*(\Omega)$ - discussed in Section~\ref{S: Examples of Regular Operators on D*(Omega): Solvable Partial Differential Equations} - can be followed by readers \emph{without strong (or any) background in Schwartz theory of distributions}.

In Section~\ref{S: Examples of Regular Operators on D*(Omega): Solvable Partial Differential Equations} we identify several subclasses of linear partial differential operators in mathematics (H\"{o}rmander~\cite{lHor-I}-\cite{lHor-IV}) as \emph{regular} (thus, surjective on $\mathcal D^*(\Omega)$) which include the following:
\begin{itemize}
	\item All linear partial differential operators \emph{with constant coefficients} are regular.
	
	\item  The \emph{H. Lewy~\cite{hLewy57} operator}: $L^*(x, \partial)=\frac{\partial}{\partial x_1}+i \frac{\partial}{\partial x_2}-2i(x_1+ix_2)\frac{\partial}{\partial x_3}$ is regular.
	
	\item \emph{All second order elliptic operators with $\mathcal C^\infty$-coefficients} are regular.
	
	\item \emph{All elliptic operators with analytic coefficients} are regular.
\end{itemize} 
All of these operators are surjective on $\mathcal D^*(\Omega)$, but not necessarily surjective on the following three invariant subspaces $\mathcal D(\Omega)$, $\mathcal E(\Omega)$ and $\mathcal D^\prime(\Omega)$ (Section~\ref{S: Three Invariant Subspaces}).
Consequently, we prove the solvability of the  partial differential equations of the form $P^*(x, \partial)U=T$ in $\mathcal D^*(\Omega)$, for regular operators $P^*(x, \partial): \mathcal D^*(\Omega)\mapsto \mathcal D^*(\Omega)$. In other words, we prove the existence of a solution $U$ in  $\mathcal D^*(\Omega)$ for every choice of $T$ also  in $\mathcal D^*(\Omega)$. 
We should recall that:
\begin{itemize}
	
	\item Every linear partial differential operator with constant coefficients $P^*(\partial)$ on $\mathcal D^\prime(\R^d)$  is surjective; this is the famous existence theorem of Malgrange~\cite{bMal55-56} and Ehrenpreis~\cite{lEhren56}. 
	
	\item The Malgrange-Ehrenpreis existence result might, however, fail in $\mathcal D^\prime(\Omega)$ 
	for operators which are  \emph{hypoelliptic but not elliptic}, subsets $\Omega$ of $\R^d$ which are \emph{open, but not $P$-convex for supports} (H\"{o}rmander~\cite{lHor-II}, Theorem 10.6.6, Corollary 10.6.8). Thus, the partial differential equation $P^*(\partial)U=f$ might have no solutions in $\mathcal D^\prime(\Omega)$ even for some smooth $f$.
	
	\item Hans Lewy~\cite{hLewy57} was the first to show that the Lewy operator $L^*(x, \partial)$ is not surjective on  $\mathcal D^\prime(\R^3)$. Thus, the partial differential equations of the form $L^*(x, \partial)U=\varphi$  might fail to have solution $U$ in $\mathcal D^\prime(\R^3)$ even for some test-functions $\varphi\in\mathcal D(\R^3)$. A general existence result also fails in the space of hyperfunctions (Schapira~\cite{pSchapira67}).
	
	\item The elliptic operators mentioned above are,  in general, also non-surjective on $\mathcal D^\prime(\Omega)$.
\end{itemize}

In Section~\ref{S: Standardization of a Non-Standard Result}  we show - with the help of Hamel bases - that the space of \emph{generalized distributions} $\widehat{\mathcal E}(\Omega)$ introduced in (Todorov~\cite{tdTod95}, \S 2) can be embedded as a $\C$-vector subspace into the algebraic dual $\mathcal D^*(\Omega)$ of the space of test-functions $\mathcal D(\Omega)$. Because $\widehat{\mathcal E}(\Omega)$ was defined in the framework of non-standard analysis (Robinson~\cite{aRob66}), we look upon $\mathcal D^*(\Omega)$ as a  \emph{standardization} of $\widehat{\mathcal E}(\Omega)$. Actually, our article itself can be viewed as a \emph{standardization} of the results in Todorov~\cite{tdTod95}, because the surjectivity of the regular operators was first proved in Todorov~\cite{tdTod95} in the framework of $\mathcal L(\widehat{\mathcal E}(\Omega))$, while the main result of this article (Theorem~\ref{T: The Main Result}) holds within $\mathcal L(\mathcal D^*(\Omega))$. Thus, by replacing $\widehat{\mathcal E}(\Omega)$ with $\mathcal D^*(\Omega)$, our result about the regular operators becomes accessible even for readers without background in non-standard analysis. Our \emph{standardization} is, of course, not an isolated event in mathematics; we remind two more  cases of standardizations in the history of mathematics.

Our \emph{inspiration comes from the fundamental theorem of algebra}: following this analogy the space $\mathcal D^\prime(\Omega)$ is the counterpart of the  \emph{field of real numbers} $\R$, the space $\mathcal D^*(\Omega)$ is the counterpart of the field of complex numbers $\C$, and the class of \emph{regular operator} is the counterpart of the \emph{ring of polynomials} $\C[x]$. We are trying to convince the reader that the space $\mathcal D^*(\Omega)$ - rather than $\mathcal D^\prime(\Omega)$ - deserves to be considered as the natural theoretical framework for the class of regular operators $P^*(x, \partial)$, since the equations of the form $P^*(x, \partial)U=T$ often have no solutions in $\mathcal D^\prime(\Omega)$. 

Recall as well that the global solvability of arbitrary analytic partial differential equations was studied in  (Rosinger~\cite{eRosinger90}, Chapter 2) and (Oberguggenberger~\cite{mOber92}, Section 22). The existence results for continuous partial differential operators are obtained by means of the Dedekind order completion method in Oberguggenberger \& Rosinger~\cite{OberRosinger94}.

A general solvability of evolution-type equations appears in Colombeau, Heibig and Oberguggenberger~\cite{ColHeibigMO93}, using \emph{regularized derivatives} in the framework of \emph{Colombeau algebra of generalized functions} (see Colombeau \cite{jCol85} and Oberguggenberger~\cite{mOber92}).  

As we mentioned above, the surjectivity of the regular operators was first proved in Todorov~\cite{tdTod95} in the framework of non-standard analysis. Meanwhile (in the period between the publication of Todorov~\cite{tdTod95} and the writing of this article) two more similar articles in the framework of $\mathcal D^*(\Omega)$ appeared: an unpublished manuscript Oberguggenberger~\& Todorov~\cite{MO&TDT} and Oberguggen-berger~\cite{MO13}. In this article we shall use some of the results in Oberguggenberger~\cite{MO13}.

On the topic of  \emph{solvability of differential equations} we refer to a relatively recent survey in Dencker~\cite{nDencker} (no connection with the space  $\mathcal D^*(\Omega)$).

Finally, we should mention that our article has somewhat an \emph{ideological edge} because we challenge at least \emph{two widely spread prejudices} in the mathematical community: \emph{The first one} is that \emph{Hamel bases are not and can never be mathematically useful}. \emph{The second one} is that we \emph{should never go beyond the space of Schwartz distributions} $\mathcal D^\prime(\Omega)$ as a framework of a partial differential equation, especially if the equation is linear. That is to say ``better to admit (perhaps with some regret) that a given equation has no solutions rather than look for a solution outside $\mathcal D^\prime(\Omega)$''.

\section{Notations and Set-Theoretical Framework}\label{S: Notations}

The \emph{set-theoretical framework} of this text  is the usual ZFC-axioms (Zermelo-Fraenkel axioms with the axiom of choice) along with the GCH (Generalized Continuum Hypothesis) in the form $2^\kappa=\kappa_+$ for every cardinal $\kappa$ (or equivalently, $2^{\aleph\alpha}=\aleph_{\alpha+1}$ for all ordinals $\alpha$). Here  we write $\kappa_+$ for the successor of $\kappa$. For the \emph{domain of ZFC and GCH axioms} we use the \emph{superstructures $\widehat{\mathcal S}$ with the set of individuals} $\mathcal S=\mathbb K\cup V$, where $V$ is the vector space over a field $\mathbb K$ under consideration (e.g. $V=\R^n$ with $\mathbb K=\R$ or $V=\C^n$ with $\mathbb K=\C$, etc.). Our formal language is based on bounded quantifiers of the form $(\forall x\in A)\alpha(x)$ and $(\exists x\in B)\beta(x)$, where $A, B\in \widehat{\mathcal S}\setminus \mathcal S$ and $\alpha(x)$ and $\beta(x)$ are predicates (Davis~\cite{mDavis}, p.11-15). We believe however, that the rest of this text can be followed without a familiarity with the concept of \emph{superstructure}.

We recall that the following are equivalent (Wolf~ \cite{rsWolf}, p. 255) and/or (Jech~\cite{tjJech}, Ch. 11, A1):
\begin{itemize}
	\item Axiom of Choice.
	\item Zorn's Lemma.
	\item Every set can be well-ordered.
	\item The usual partial order on the class of cardinal numbers is a total order. 
\end{itemize}
In particular, Zorn's Lemma will be involved in Theorem~\ref{T: Existence of Basis} and the total order between cardinals is needed in the proof of   Lemma~ \ref{L: Subspace Lemma}. Also, "every set can be well-ordered"   will be useful to supply a basis with well-ordering if desired (Remark~\ref{R: Well-Ordered Bases}). 

Actually, we do not need the GCH except for the purpose of simplifying the calculations with cardinals and the dimension of vector spaces. For example,  with the help of GCH,  $\aleph_0<\dim(V)\leq\mathfrak{c}$ implies simply $\dim(V)=\mathfrak{c}$ (rather than only $\dim(V)=\aleph_1$). Here $\aleph_0=\card(\mathbb N)$ and $\mathfrak{c}=\card(\R)$. 

If $X$ is a set, we shall treat $X$ as a subset of the \emph{power set} $\mathcal P(X)$, in symbols, $X\subset\mathcal P(X)$ by means of the embedding $x\mapsto \{x\}$. If $X$ and $Y$ are two sets, we denote by $Y^X$ the set of all functions from $X$ to $Y$. 

For \emph{index sets} (for indexing bases, for example)  we use the popular sets: $\mathbb N, \R, \R^d,$ $\mathcal P(\R),  \mathcal P(\R^d), \mathcal P(\mathcal P(\R))$, etc. with cardinalities $\aleph_0, \mathfrak{c}, \mathfrak{c}, \mathfrak{c}_+, \mathfrak{c}_+$, $(\mathfrak{c}_+)_+$, respectively. We sometimes use the field of scalars $\mathbb K$ itself as an index set or $\mathbb K^d, \mathcal P(\mathbb K),  \mathcal P(\mathbb K^d), \mathcal P(\mathcal P(\mathbb K))$, etc.

In what follows $V$ stands for a generic \emph{vector space over a field} $\mathbb K$ (Axler~\cite{sAxlerBook}). Sometimes we shall write  $V|\,\mathbb K$ instead of $V$. If we write  $U\overset{\centerdot}{\subseteq} V$, we mean that both $U$ and $V$ are vector spaces over the same field and $U$ is a vector subspace of $V$. Similarly,  $V \cong W$ means that $V$ and  $W$ are isomorphic vector spaces. $\mathcal L(V)$ denotes the $\mathbb K$-vector space consisting of all linear operators $L: V\mapsto V$. We denote by $V^*$ the algebraic dual of $V$. We denote by $T(v)$ or $\langle T, v\rangle$ the evaluation of  $T\in V^*$ at $v\in V$.

Let $\mathcal T^d$ denote the \emph{usual topology} on $\R^d$ and let $X, Y\in\mathcal T^d$ be two open set of $\R^d$. We denote by $\Diff(X, Y)$ the set of all \emph{diffeomorphisms} from $X$ to $Y$. If $\theta\in\Diff(X, Y)$, we denote by  $J_\theta: X\to\R$,  $J_\theta=\big|\det\big(\frac{\partial\theta}{\partial x}\big)\big|$, the corresponding \emph{Jacobian determinant}. We denote by $\Diff(X)$ the group of  \emph{diffeomorphisms} from $X$ to itself.

Let $\Omega$ stand for a (generic) open subset of $\R^d$. Here is a list of popular functional spaces and notations:

\begin{itemize}
	\item  $\mathcal E(\Omega)=\mathcal C^\infty(\Omega)$ denotes the space $\mathcal C^\infty$-functions from $\Omega$ to $\C$.
	
	\item $\mathcal D(\Omega)=\mathcal C_0^\infty(\Omega)$ denote the \emph{space of test-functions} on $\Omega$, i.e. the $\mathcal C^\infty$-functions from $\Omega$ to $\C$ with compact support (Vladimirov~\cite{vVladimirov}). 
	
	\item $\mathcal L^2(\Omega)$ denotes the usual Hilbert space of Lebesgue measurable square integrable functions from $\Omega$ to $\C$.
	
	\item $\mathcal L^\infty(\Omega)$ denotes the space of Lebesgue measurable bounded functions from $\Omega$ to $\C$.
	
	\item $\mathcal L_{loc}(\Omega)$ stands for the  Lebesgue measurable locally  integrable  functions from $\Omega$ to $\C$.

	\item $\mathcal D^\prime(\Omega)$ denotes the \emph{space of Schwartz distributions} on $\Omega$ (Vladimirov~\cite{vVladimirov}).
	
	\item $\mathcal E^\prime(\Omega)$ denotes the space of Schwartz distributions with compact support. 
	
	\item We denote by $\mathcal D^*(\Omega)$ and $\mathcal E^*(\Omega)$ the algebraic duals of $\mathcal D(\Omega)$ and $\mathcal E(\Omega)$, respectively.
\end{itemize}

\section{Infinite-Dimensional Linear Algebra}\label{S: Infinite-Dimensional Linear Algebra}

We recall the definitions of \emph{algebraic (Hamel) basis} and \emph{dimension} of an infinite-dimensional vector space (Hamel~\cite{gHamel}). For more details and the missing proofs we refer to Jacobson~\cite{nJacobsonLectures}, Hungerford~\cite{t.w.Hungerford}, Hewitt \&. Stromberg~\cite{HewittStromberg65} and Mackey~\cite{gwMackey45}.

\begin{definition}[Basis and Dimension] \label{D: Basis and Dimension} Let $V$ be a non-trivial vector space over a field $\mathbb K$. 
	\begin{D-enum}
		\item A subset $\mathcal B$ of $V$ is called \emph{free}  if every finite subset of $\mathcal B$ consists of linearly independent vectors in $V$.
		
		\item A free set $\mathcal B$ of $V$ is called \emph{maximal} (or, a \emph{maximal free set}) if $\mathcal B$ cannot be extended (properly) to a free set of $V$. Every maximal free set $\mathcal B$ of $V$ is called a \emph{algebraic basis}, \emph{Hamel basis} or simply, \emph{basis} of $V$. 
		
		\item If $\mathcal B$ is a basis of $V$, the \emph{dimension} (or \emph{Hamel dimension}) of $V$ is defined by  $\dim V=\card \,\mathcal B$.
	\end{D-enum}
\end{definition}

The definition of \emph{dimension} is justified by the following result.

\begin{lemma}[Justification] \label{L: Justification} All bases of $V$ have the same cardinality.
\end{lemma}

\begin{proof} We refer to (Jacobson~\cite{nJacobsonLectures}, Chapter 9, \S 2, p.240) or  (Hungerford~\cite{t.w.Hungerford}, Theorem 2.6, p.184), or  (Hewitt \&. Stromberg~\cite{HewittStromberg65}, Thm 4.58, p. 30).
\end{proof}

\begin{lemma}[Span]\label{L: Span} Let $\mathcal B$ be a basis of $V$. Then every non-zero vector $v\in V$ can be uniquely presented as a (finite) linear combination of vectors in $\mathcal B$ with non-zero coefficients in $\mathbb K$. We summarize this in $V={\rm span}\,\mathcal B$. Consequently, $\card\,V=\max\{\dim V,\, \card\, \mathbb K\}$.
\end{lemma}

\begin{proof} Let $v\in V, v\not=0$, and suppose (seeking a contradiction) that $v\not=\sum_{k=1}^m c_n v_n$ for all $m\in\mathbb N$, all $v_n\in\mathcal B$ and all $c_n\in\mathbb K$. Thus the set $\overline{\mathcal B}=\mathcal B\cup\{v\}$ is also a free set, contradicting the maximality of $\mathcal B$.
\end{proof}

\begin{definition}[Spectrum]\label{D: Spectrum} Let $\mathcal B=\{v_s: s\in S\}$ be a Hamel basis of $V$ with an index set $S$ with $\card\, S=\dim V$. Let $v\in V, v\not=0$. Then the (finite) set $\spm(v)\subset S$  is called the \emph{spectrum} of $v$ (relative to $\mathcal B$ and $S$) if $v=\sum_{s\in\spm (v)} c_s v_s$ and $c_s\in\mathbb K\setminus\{0\}$ for all $s\in\spm (v)$. We shall sometimes write simply $v=\sum_{s\in S} c_s v_s$ or even $v=\sum c_s v_s$ instead of $v=\sum_{s\in\spm (v)} c_s v_s$ (keeping in mind that the sum $\sum_{s\in S} c_s v_s=\sum c_s v_s$ is finite). We also let $\spm(0)=\varnothing$ and $\sum_{s\in \varnothing} c_s v_s=0$.
\end{definition}

Notice that $\card\, \mathbb K\leq\card\, V$ and $\dim V\leq\card\, V$ hold trivially. The next equalities follow immediately from the formula	in Lemma~\ref{L: Span}.

\begin{corollary}[Two Equalities]\label{C: Two Equalities} 
	\begin{C-enum}
		\item If $\card\,\mathbb K<\card\,V$, then $\dim V
		=\card\,V$ (see Example~\ref{Ex: Hamel Example} in this paper).
		
		\item If $\dim V<\card\,V$, then $\card\, V=\card\,\mathbb K$ (see Example~\ref{Ex: The Space of Polynomials and its Dual} and Example~\ref{Ex: The Space of Non-Standard Polynomials and its Dual}).
	\end{C-enum}
\end{corollary}	

\begin{theorem}[Existence of Basis]\label{T: Existence of Basis} Let $V$ be a vector space over a field of scalars $\mathbb K$ and let $E\subset V$ be a free set of $V$. Then there exists a basis $\mathcal B$ of $V$ which contains $E$ and such that $V= {\rm span}\, E\oplus{\rm span} (\mathcal B\setminus E)$. Consequently, every non-trivial vector space has a basis.
\end{theorem}
\begin{proof} Consider the family of subsets of $V$:
	\[
	\mathcal F(E)= \{S\in\mathcal P(V): S \text{ is a free set and } E\subseteq S\}.
	\]
	We shall treat $\mathcal F(E)$ as a partially ordered set under the inclusion, $\subseteq$. Note that $\mathcal F(E)$ is a non-empty set, because $E\in\mathcal F(E)$. We observe that every totally ordered subset (chain) $\mathcal C$ of $\mathcal F(E)$ is bounded from above by its union $\bigcup_{C\in\mathcal C} C$ and also $\bigcup_{C\in\mathcal C} C\in\mathcal F(E)$. By \emph{Zorn's Lemma} $\mathcal F(E)$ has a maximal element, $\mathcal B$. 
\end{proof}

\begin{remark}[Hamel Bases] \label{R: Hamel Bases}
	\begin{R-enum}
		\item We sometimes refer to the \emph{maximal free sets} $\mathcal B$ of $V$ as   \emph{Hamel bases} of $V$ after Georg Hamel~\cite{gHamel} who proved the above theorem in 1905 in the particular case of $V=\R$ and $\mathbb K=\mathbb Q$ (see Example~\ref{Ex: Hamel Example} later in the paper). 
		
		\item Unlike the case of finite-dimensional vector spaces, in the case of an infinite-dimensional vector space $V$ the equality $\card\,E=\dim V$ for some free set $E$ of $V$ does not imply that $E$ is a basis of $V$. Indeed, let $\mathcal B$ be a basis of $V$ and let $E=\mathcal B\setminus\{v\}$ for some $v\in\mathcal B$. Then $E$ is a free set with $\card\,E =\dim V$, but $E$ is not a basis for $V$.
	\end{R-enum}
\end{remark}
The next result validates the usefulness of the notion of \emph{Hamel dimension}.

\begin{theorem}[Isomorphic Spaces]\label{T: Isomorphic Spaces} Let $V$ and $W$ be two vector spaces over the same field $\mathbb K$ such that $\dim V= \dim W$. Then $V$ and $W$ are isomorphic. In particular, the mapping $\sigma: V\mapsto W$, defined by $\sigma(\sum_{s\in S}c_s v_s)=\sum_{s\in S}c_s w_s$, is a vector-isomorphism from $V$ to $W$, where $(v_s)_{s\in S}$ and  $(w_s)_{s\in S}$ are bases of $V$ and $W$, respectively, $S$ is an index set of $\card\, S=\dim V= \dim W$ and $c_s\in\mathbb K$ for all $s\in S$.
\end{theorem}
\begin{proof} The proof is almost identical to the proof of finite-dimensional case and we leave it to the reader.
\end{proof}

\begin{lemma}[Subspace Lemma] \label{L: Subspace Lemma} Let $U$ and $V$ be two vector spaces over the same field $\mathbb K$. Then either $U$ and $V$ are isomorphic, or one of the spaces is (can be embedded as) a subspace of the other (see Remark~\ref{R: Warning}). Consequently, if $U$ is a vector subspace of $V$ and $\dim U<\dim V$, then $U$ is a proper subspace of $V$.
\end{lemma} 

\begin{proof} Let $\dim U=\alpha$ and  $\dim V=\beta$. Then exactly one of the following holds: $\alpha=\beta$, $\alpha<\beta$,  $\alpha>\beta$, by the axiom of choice in its forth version (Section~\ref{S: Notations}).
\end{proof}

\begin{remark}[Warning] \label{R: Warning} If $U$ and $V$ are infinite-dimensional vector spaces over the same field of scalars $\mathbb K$, it might happen that $U$ is a proper subspace of $V$ and at the same time $\dim U=\dim V$. For example, let $\mathbb K=\R$ and $V$ be the vector space $\R^\infty$ consisting of all sequences in $\R$ with finite support. Let $\{(e_1), (e_2),\dots\}$ be the standard basis of $\R^\infty$, i.e. $(e_n)$ is a sequence in $\R$, defined by  $(e_n)_i=\begin{cases}1, &\text{if $n=i$,}\\
		0, &\text {if $n\not=i$}.
	\end{cases}$ (for more detail, we refer to Section~\ref{S: Coordinate Isomorphism}). Then $U={\rm span}\{e_2, e_3,\dots\}$ is obviously a proper subspace of $\R^\infty$. On the other hand, $L\in\mathcal L(U, V)$, defined by $L(e_{n+1})=e_n, n\in\mathbb N$, is an isomorphism from $U$ to $V$.
\end{remark}
\begin{corollary}[Algebraic Complement]\label{C: Algebraic Complement} Every subspace $U$ of $V$ has a (non-unique) algebraic complement $W$ to $V$, i.e. a subspace  $W$ of $V$ such that $V=U\oplus W$.
\end{corollary}

\begin{proof}  Let $A\subseteq B$ hold for two (index) sets with $\card(A)=\dim(U)$ and $\card\,B=\dim V$. Let $\{u_\alpha: \alpha\in A\}$ be a (Hamel) basis of $U$ and 
	\[
	\{u_\alpha: \alpha\in A\}\cup\{w_\beta: \beta\in B\setminus A\},
	\]
	be its extension to a Hamel basis of $V$ (Theorem~\ref{T: Existence of Basis}). Then $W={\rm span}\{w_\beta: \beta\in B\setminus A\}$.
\end{proof}

The next result is in sharp contrast to its counterpart in the finite-dimensional linear algebra.

\begin{lemma}[Realification]\label{L: Realification} Let $V|\mathcal R(i)$ denote the vector space
	$V$ over a field of the form $\mathcal R(i)$, where $\mathcal R$ is a formally real (orderable) field. Let $V|\mathcal R$ denote the realification (decomplexification) of $V|\mathcal R(i)$. Then if one of the vector space is infinite-dimensional, so is the other and  $\dim (V|\mathcal R(i))=\dim (V|\mathcal R)$. 
\end{lemma}
\begin{proof} If $\{v_s: s\in S\}$ is a Hamel basis of  $V|\mathcal R(i)$, then $\{v_s+i v_s: s\in S\}$ is a Hamel basis of $V|\mathcal R$ and $\card(\{v_s: s\in S\})=\card(\{v_s+i v_s: s\in S\})$.
	
\end{proof}

\section{Algebraic Dual}\label{S: Algebraic Dual}

We shortly discuss the properties of the algebraic dual $V^*$ of an infinite-dimensional vector space $V$. Both $V^*$ and $V^{**}$ are proper vector space extensions  of $V$. In sharp contrast to the finite-dimensional case however, the vector spaces $V$, $V^*$ and $V^{**}$ are never isomorphic.

\begin{theorem}[Dimension of Dual Space]
	\label{T: Dimension of Dual Space} Let $V$ be an infinite dimensional vector space over the (infinite) field $\mathbb K$ and $V^*$ denotes the algebraic dual of $V$. Then $\dim V^*=\max\big\{2^{\dim V},\;  \card\,\mathbb K\big\}=\max\big\{(\dim V)_+,\;  \card\,\mathbb K\big\}$ (the  formula fails for finite-dimensional $V$). Consequently, $\dim V^* > \dim V$ for any infinite-dimensional vector space $V$.
\end{theorem}

\begin{proof}  We start from the formula $\dim V^*=(\card\, \mathbb K)^{\dim V}$ derived in (Jacobson~\cite{nJacobsonLectures}, Chapter 9, \S 5, p.245). Next, assuming ZFC+GCH (Section~\ref{S: Notations}), we show that $y^x=\max\{y, 2^x\}=\max\{y, x_+\}$
	for every two infinite cardinals,  $x$ and $y$. Indeed, if $y=\aleph_0$, the formula follows from the fact that $\aleph_0< 2^x$. Let $y$ be uncountable. Then $y=2^\kappa$ for some infinite cardinal $\kappa$ by the GCH. Thus $y^x=(2^\kappa)^x=2^{\kappa x}=2^{\max\{\kappa, x\}}=\begin{cases} 2^\kappa, &\text{if $\kappa\leq x$;}\\
		2^x, &\text{if $\kappa\geq x$}
	\end{cases}=
	\begin{cases} y, &\text{if $y\leq 2^x$;}\\
		2^x, &\text{if $y\geq 2^x$}
	\end{cases}=\max\{y, 2^x\}=\max\{y, x_+\}$.
	Finally, we let $y=\card(\mathbb K)$ and $x=\dim V$. The second equality in the above formula follows from the first equality since $2^{\dim V}=(\dim V)_+$ by the GCH (Section~\ref{S: Notations}).
\end{proof}

Notice that $\dim V^*\geq\card\, \mathbb K$ and $\dim V^*\geq(\dim V)_+$  hold trivially. The next equalities follow immediately from the formula	in Theorem~\ref{T: Dimension of Dual Space}.

\begin{corollary}[Two More Equalities]\label{C: Two More Equalities}
	\begin{C-enum}
		
		\item If $\card\,\mathbb K<\dim V^*$, then \newline
		$\card\, V^*=\dim V^*=(\dim V)_+$ (see Example~\ref{Ex: Hamel Example} in this paper).
		
		\item  If $(\dim V)_+<\dim V^*$, then $\dim V^*=\card\, \mathbb K$ (see Example~\ref{Ex: The Space of Non-Standard Polynomials and its Dual}). 
	\end{C-enum}
	
\end{corollary}
\begin{definition}[Embeddings \& Restricted Duals]\label{D: Embeddings & Restricted Duals} 
	\begin{D-enum}
		\item Let $\mathcal B=\{v_s: s\in S\}$ be a basis of $V$ indexed by a set $S$. Let $\{\Phi_r: r\in S\}$ be a subset of $V^*$ defined by $
		\Phi_r(v_s)=\begin{cases} 1 &\text{if $r=s$},\\
			0 & \text{if $r\not=s$}.
		\end{cases}$
		The subspace ${V^*_\mathcal B}={\rm span}\{\Phi_r: r\in S\}$ of  $V^*$ is the \emph{restricted dual} of  $V$ relative to $\mathcal B$. The mapping $\sigma_\mathcal B: V\mapsto V^*$ (with range $\sigma[V]=V^*_\mathcal B$), defined by $\sigma_\mathcal B(v_s)=\Phi_s$ for all $s\in S$, is the \emph{vector space embedding} of $V$ into $V^*$ \emph{relative to $\mathcal B$}. We write this as  $V\subset_\mathcal B V^*$. If $\mathcal B$ is a standard Hamel basis of $V$, we shall write $\sigma$, $V_*$ and $V\subset V^*$ instead of $\sigma_\mathcal B$, $V^*_\mathcal B$ and $V\subset_\mathcal B V^*$, respectively (for an example we refer to \# 2 and \# 3 in Definition~\ref{D: The Space K0S}).
		
		\item The mapping $\iota: V\mapsto V^{**}$, defined by $\iota(v)(T)=T(v)$ for all $T\in V^*$, is called \emph{canonical embedding} of $V$ into $V^{**}$ (it does not depend on a choice of any basis). We write simply, $V\subset V^{**}$.
	\end{D-enum}
\end{definition}
\begin{corollary}[Embedding of Duals]\label{C: Embedding of Duals} Let $U$ be a subspace of $V$ and $W$ be an algebraic complement of $U$ to $V$, i.e. $V=U\oplus W$ (Corollary~\ref{C: Algebraic Complement}). Let $\sigma_W: U^*\mapsto V^*$ be the mapping defined by $\sigma_W(T)(v)=T(u)$ for all $v\in V$, where $v=u+w$, $u\in U$ and $w\in W$. Then $\sigma_W$ is a vector space embedding of $U^*$ into $V^*$. We denote this by $U^*\subseteq_W V^*$  or even by $U^*\subseteq V^*$ (if $W$ is understood). 
\end{corollary}

For a recent study, from a purely algebraic point of view, of the relationship between the \emph{restricted dual} $V_*$ and the \emph{algebraic dual} $V^*$ of a vector space $V$ with a countable Hamel basis, we refer to the recent article Chirvasitu \& Penkov~\cite{Chiv&Penkov} (no relation to solvability of PDE and generalized functions).

\section{Linear Maps and Operators}\label{S: Linear Maps and Operators}

We present selected results of linear algebra (needed for the rest of the article) which are well-known for finite-dimensional vector spaces, but less-known for infinite-dimensional spaces.

\begin{theorem}[Extension Principle]\label{T: Extension Principle}
	Let $U, V$ and $W$ be three vector spaces over the same field of scalars, $\mathbb K$ (the case $W=\mathbb K$ is not excluded) and let $U$ be a subspace of $V$. Then every linear map $L\in\mathcal L(U, W)$ can be extended (non-uniquely) to a linear map $\widehat{L}\in\mathcal L(V, W)$.
\end{theorem}
\begin{proof} Let $A\subseteq B$ hold for two (index) sets with $\card(A)=\dim(U)$ and $\card(B)=\dim(V)$. Let $\{u_\alpha: \alpha\in A\}$ be a (Hamel) basis of $U$ and $
	\{u_\alpha: \alpha\in A\}\cup\{v_\beta: \beta\in B\setminus A\}$
	be its extension to a (Hamel) basis of $V$ (Theorem~\ref{T: Existence of Basis}). Let $\widehat{L}: V\mapsto W$ be a map defined by $\widehat{L}(u_\alpha)=L(u_\alpha)$ for $\alpha\in A$ and anyhow (for example, $\widehat{L}(v_\beta)=0$) for $\beta\in B\setminus A$. Then $\widehat{L}$  is an extension of $L$ we are looking for.
\end{proof}
\begin{definition}[Duals and Co-Duals]\label{D: Duals and Co-Duals} Let $\mathcal O\in\mathcal L(V)$ and  $\mathcal O^*\in\mathcal L(V^*)$ be two operators such that  $\mathcal O^*(T)=T\, \circ\, \mathcal O$ for all $T\in V^*$. Then we say that the operator $\mathcal O^*$ is the \emph{dual} of $\mathcal O$ and that  $\mathcal O$ is the \emph{co-dual} (or \emph{transposed}) of $\mathcal O^*$. We sometimes use the notation $\mathcal O= {^t(\mathcal O^*)}$.
\end{definition}

\begin{remarks}[Bracket Notation]\label{R: Bracket Notation} We often write $\bra T, v\ket$ instead of $T(v)$ for the evaluation of $T\in V^*$ at $v\in V$. In this \emph{bracket notation} the above definition can be summirized as follows: $\bra \mathcal O^*(T), v\ket=\bra T, \mathcal O(v)\ket$ for all $v\in V$ and all $T\in V^*$. 
\end{remarks}

\begin{theorem}[Surjective Dual]\label{T: Surjective Dual}  Let $V$ be  a vector space and $V^*$ be its (algebraic) dual. Then $\mathcal O\in\mathcal L(V)$ is injective \ifff  its dual $\mathcal O^*\in\mathcal L(V^*)$ is surjective. 
\end{theorem}

\begin{proof} 
	\begin{Pf-enum}
		
		\item[($\Rightarrow$)] Let $T \in V^*$. We have to show that the equation $\mathcal O^*(\Lambda)=T$ has a solution $\Lambda$ in $V^*$. Indeed, define $\Phi: \ran(\mathcal O)\mapsto\mathbb K$ by $\Phi(\mathcal O(v))=T(v)$ for all $v\in V$, where $ \ran(\mathcal O)$ stands for the range of $\mathcal O$. We observe that $\Phi$ is well-defined, because $\mathcal O$ is injective  by assumption. It is clear that $\Phi\in (\ran(\mathcal O))^*$, because $T$ is linear. By the Extension Principle (Theorem~\ref{T: Extension Principle}), $\Phi$ can be extended to some $\Lambda\in V^*$. Thus $\mathcal O^*(\Lambda)(v)=\Lambda(\mathcal O(v))=\Phi(\mathcal O(v))=T(v)$ for all $v\in V$. So, $\mathcal O^*(\Lambda)=T$, as required. 
		
		\item[($\Leftarrow$)] Let $\mathcal O(v)=0$ for some $v\in V$. We have to show that $v=0$. Indeed, let $T\in V^*$. Since $\mathcal O^*$ is surjective, there exists $\Lambda\in V^*$  such that $T= \mathcal O^*(\Lambda)$. Thus $T(v)= \mathcal O^*(\Lambda(v))=\Lambda(\mathcal O(v))=\Lambda(0)=0$. Thus $T(v)=0$ for all $T\in V^*$ implying $v=0$.
	\end{Pf-enum}
\end{proof}

The above result gives rise to the concept of a \emph{regular operator} (used in Todorov~\cite{tdTod95} in the particular case of $V=\mathcal D(\R^d)$ (Example~\ref{Ex:  D(Omega) and its Dual}).
\begin{definition}[Regular Operators]\label{D: Regular Operators} An operator $\mathcal O^*$ in $\mathcal L(V^*)$ is \emph{regular}  if $\mathcal O^*$ has an injective co-dual (transposed)  $\mathcal O\in\mathcal L(V)$. 
\end{definition}

\begin{corollary}[Solvability]\label{C: Solvability}  Let $V$ be  a vector space and $V^*$ be its (algebraic) dual. Let $\mathcal O^*\in\mathcal L(V^*)$ be a regular operator. Then the equation $\mathcal O^*(\Lambda)=T$ is \textbf{solvable} in $V^*$  in the sense that for every choice of $T\in V^*$  there exists $\Lambda\in V^*$ such that $\mathcal O^*(\Lambda)=T$.
\end{corollary}
\begin{proof} An immediate consequence of Theorem~\ref{T: Surjective Dual}.
\end{proof}

\section{Coordinate Isomorphism}\label{S: Coordinate Isomorphism}

We discuss vector spaces $\mathbb K_0^S$, which are  infinite-dimensional counterpart of the familiar vector spaces $\mathbb K^d$.

\begin{definition}[The Space $\mathbb K_0^S$] \label{D: The Space K0S} Let $\mathbb K$ be a field and $S$ be a non-empty set (well-ordered if desired).  
	\begin{D-enum}
		\item We denote by $\mathbb K_0^S$ the set of all functions $f:S\mapsto\mathbb K$ with finite support $\textrm{supp}(f)=\{s\in S: f(s)\not=0\}$.
		
		\item For each $s\in S$ we define the net $e_s: S\mapsto\mathbb K$ by $
		e_s(t)=\begin{cases}1, &\text{if $s=t$,}\\
			0, &\text {if $s\not=t$},
		\end{cases}$ or simply, $e_s(t)=\delta_{st}$ for short. We refer to the set $\{e_s: s\in S\}$  as the \emph{standard (Hamel) basis} of $\mathbb K_0^S$.
		
		\item Let $(\mathbb K_0^S)^*$ be the (algebraic) dual of $\mathbb K_0^S$ and let $\{\Phi_r: r\in S\}$ be a subset of $(\mathbb K_0^S)^*$ defined by $
		\Phi_r(e_s)=\delta_{rs}$.
		The subspace ${(\mathbb K_0^S)_*}={\rm span}\{\Phi_r: r\in S\}$ of  $(\mathbb K_0^S)^*$ is the \emph{restricted dual} of  $\mathbb K_0^S$ (Definition~\ref{D: Embeddings & Restricted Duals}). The mapping $\sigma: \mathbb K_0^S\mapsto (\mathbb K_0^S)^*$, defined by $\sigma(e_s)=\Phi_s$ for all $s\in S$, is the \emph{vector space embedding} of $\mathbb K_0^S$ into $(\mathbb K_0^S)^*$, which will be written simply as  $\mathbb K_0^S\subset(\mathbb K_0^S)^*$.
	\end{D-enum}
\end{definition} 
\begin{theorem}[Properties of $\mathbb K_0^S$]\label{T: Properties of K0S} Let $\mathbb K$ be a field and $S$ be a non-empty set (as above). Then:
	\begin{T-enum}
		\item $\mathbb K_0^S$ is a \emph{vector space over} $\mathbb K$ and $\{e_s: s\in S\}$ is a bases for $\mathbb K_0^S$. Consequently, $\dim\mathbb K_0^S=\card\, S$ and  $\card\, \mathbb K_0^S=\max\{\card\, S, \card\,\mathbb K\}$.
		\item $\dim(\mathbb K_0^S)^*=\card{(\mathbb K_0^S)^*}=\max\{(\card\, S)_+, \card\,\mathbb K\}$.
	\end{T-enum}
\end{theorem}
\begin{proof} The part (i) follows immediately from the definition of the spaces $\mathbb K_0^S$ and Lemma~\ref{L: Span}. For part (ii) we calculate: $\dim\, (\mathbb K_0^S)^*=
	\max\big\{(\dim \mathbb K_0^S)_+, \card\, \mathbb K\big\}=\max\big\{(\card\, S)_+,  \card\, \mathbb K\}$ by Theorem~\ref{T: Dimension of Dual Space}. Next, we apply Lemma~\ref{L: Span} in the particular case $V=(\mathbb K_0^S)^*$ and calculate: $ \card (\mathbb K_0^S)^*=\max\{\dim (\mathbb K_0^S)^*,  \card\, \mathbb K\}=\max\big\{(\max\big\{(\card\, S)_+,  \card\, \mathbb K\},  \card\, \mathbb K\big\}
		=\max\big\{(\card\, S)_+,  \card \,\mathbb K\}$
as required.
\end{proof}	
Notice that if $S$ is an ordered finite set,  the space $\mathbb K_0^S$ reduces to the familiar $\mathbb K^d$, where $d=\card\, S$.

\begin{theorem}[Coordinate Isomorphism]\label{T: Coordinate Isomorphism} Let $\mathbb K$ be a field, $V$ be a vector space over $\mathbb K$ (as before).  Let $S$ be a set of $\card\, S=\dim V$ and $\mathcal B=\{v_s: s\in S\}$ be a basis of $V$. Then:
	\begin{T-enum}
		\item   The mapping  $\Gamma: V\mapsto \mathbb K_0^S$, defined by $\Gamma(v)=f$, where $v=\sum_{s\in S} f(s)\, v_s$,  is a vector isomorphism. We call $\Gamma$ a \textbf{coordinate isomorphism} and $f$ the \textbf{coordinate function} of $v$ relative to the basis $\mathcal B$ (sometimes the notation $f_v$ or even $f_{v,\mathcal B}$ should be used instead of $f$). In the particular case of $S=\mathcal B$, we have  $\Gamma: V\mapsto \mathbb K_0^\mathcal B$, where $\Gamma(v)=f$ and $v=\sum_{w\in \mathcal B} f(w)\, w$. 
		
		\item For every $s\in S$ we have $\Gamma^{-1}(e_s)=v_s$, where $\{e_s: s\in S\}$ is the standard basis of $\mathbb K_0^S$ (Definition~\ref{D: The Space K0S}). 
	\end{T-enum}
\end{theorem}
\begin{proof}
	\begin{Pf-enum} 
		\item We observe that $\Gamma(v_s)=e_s$ for every $s\in S$. Thus $\Gamma\in\mathcal L(V, \mathbb K_0^\mathcal B)$ is a bijection. 
		\item follows directly from (ii).
	\end{Pf-enum}
\end{proof}
\begin{remark}[Well-Ordered Bases] \label{R: Well-Ordered Bases} 
	We often use matrices (including row and column matrices) to visualize the coordinate functions of the vectors and linear operators  relative to a particular basis of $V$. The \emph{matrix approach} is exceptionally popular in the cases of finite or countable dimensional vector spaces as well as in separable Hilbert spaces. Can we extend the \emph{matrix approach} to uncountable vector spaces? The answer is yes; we have to invoke the axiom of choice again in the form of its third version (Section~\ref{S: Introduction}): Every set, in particular every basis $\mathcal B$ of $V$, can be well-ordered. Alternatively, we can well-order the index set $S$ in $\mathbb K_0^S$. We call these bases \emph{well-ordered bases}.
\end{remark}
\section{Examples of Infinite-Dimensional Spaces}\label{S: Examples of Infinite-Dimensional Spaces}

We present several examples of infinite-dimensional vector spaces and their algebraic duals and demonstrate how to choose an algebraic (Hamel) basis (Theorem~\ref{T: Existence of Basis}). We shall often rely on the formulas: 
\begin{align}	
	&\card\, V=\max\{\dim V, \card\, \mathbb K\},\label{E: Cardinality}\\
	&\dim{V^*}=\max\{(\dim V)_+, \card\, \mathbb K\}\label{E: Dimension of Dual},\\
	&\card\,{V^*}=\max\{\dim V^*, \card\, \mathbb K\}\label{E: Cardinality of Dual},
\end{align}
(Lemma~\ref{L: Span} and Theorem~\ref{T: Dimension of Dual Space}) along with the shortcuts presented in Corollary~\ref{C: Two Equalities}  and Corollary~\ref{C: Two More Equalities}. Here is some advice for the order of the calculations: If $\dim(V)$ is known (or easy to calculate), we recommend the order of calculations: $\dim(V)\mapsto\card(V)\mapsto\dim(V^*)\mapsto\card(V^*)$.  If $\card(V)$ is known (or easy to calculate), we recommend:  $\card(V)\mapsto\dim(V)\mapsto\dim(V^*)\mapsto\card(V^*)$. Recall that $\dim(V)\leq\card(V)$ holds trivially.
\begin{example}[Hamel Example]\label{Ex: Hamel Example} Let $\R|\mathbb Q$ denote the $\mathbb Q$-vector space of $\R$ and $(\R|\mathbb Q)^*$ stand for its (algebraic) dual. We have $\dim(\R|\mathbb Q)=\card\,\R=\mathfrak{c}$ by Corollary~\ref{C: Two Equalities}, since $\card\,\mathbb Q<\card\,\R$. Also, $\dim{(\R|\mathbb Q)^*}=\max\{\mathfrak{c}_+, \aleph_0\}=\mathfrak{c}_+$ by (\ref{E: Dimension of Dual}) and $\card{(\R|\mathbb Q)^*}=\mathfrak{c}_+$ by Corollary~\ref{C: Two More Equalities}, since again, $\card\,\mathbb Q<\card\,\R$.  For the original source of this example, we refer to Hamel~\cite{gHamel}.
\end{example}

\begin{example}[The Space $\R|\mathcal A$ and its Dual]\label{Ex: R|A and its Dual} Let $\mathcal A$ denote the field of algebraic real numbers. Let $\R|\mathcal A$ denote the $\mathcal A$-vector space of $\R$ and $(\R|\mathcal A)^*$ stand for its dual. As in the previous example, we have $\dim(\R|\mathcal A)=\card\,\R=\mathfrak{c}$ and $\dim{(\R|\mathbb Q)^*}=\card{(\R|\mathcal A)^*}=\mathfrak{c}_+$ since $\card\,\mathcal A=\card\,\mathbb Q<\card\,\R$.
\end{example}

\begin{example}[$\R^\mathbb N$ and its Dual]\label{Ex: RN and its Dual} \begin{itemize}
		\item Let $\R^\mathbb N$  denote the $\R$-vector space of all sequences in $\R$ and  $(\R^\mathbb N)^*$ denote its dual space. We show that $\dim(\R^\mathbb N)=\card(\R^\mathbb N)=\mathfrak{c}$ and $\dim{(\R^\mathbb N)^*}=\card{(\R^\mathbb N)^*}=\mathfrak{c}_+$. Indeed, $\card(\R^\mathbb N)=\mathfrak{c}^{\aleph_0}= (2^{\aleph_0})^{\aleph_0}=2^{\aleph_0\aleph_0}=2^{\aleph_0}=\mathfrak{c}$ (the last step is due to the CH, Section~\ref{S: Notations}). On the other hand, $\dim(\R^\mathbb N)\leq\mathfrak{c}$ holds trivially. To show $\dim(\R^\mathbb N)\geq\mathfrak{c}$, observe that the subset $E=\{(\frac{1}{r^n}): r\in\R, r\not=0\}$ of $\R^\mathbb N$ is a free set of $\R^\mathbb N$, where $(\frac{1}{r^n})=(\frac{1}{r}, \frac{1}{r^2},\dots)$.  Thus $\dim(\R^\mathbb N)=\mathfrak{c}$ (as required) and $\R^\mathbb N\cong \R_0^{\R^2}$ by Theorem~\ref{T: Coordinate Isomorphism}, since $\card(\R^2)=\mathfrak{c}$. Consequently,  $E$ can be extended to a Hamel basis, say,  $\mathcal B=\{g_{(r,s)}: (r, s)\in\R^2,\, (r, s)\not=(0, 0)\}$ of $\R^\mathbb N$ by Theorem~\ref{T: Existence of Basis}, where $g_{(r,0)}=(\frac{1}{r^n})$ for all  $r\in\R, r\not=0$. 
		
		\item Next, $\dim(\R^\mathbb N)^*=\max\big\{\mathfrak{c}_+, \mathfrak{c}\big\}= c_+$ by (\ref{E: Dimension of Dual}) and $\card\, (\R^\mathbb N)^*=\mathfrak{c}_+$ by Corollary~\ref{C: Two More Equalities}, since $\card\, \R<\dim(\R^\mathbb N)^*$. Consequently,  $(\R^\mathbb N)^*\cong \R_0^{\mathcal P(\R^2)}$ by Theorem~\ref{T: Coordinate Isomorphism}. Let us consider the subset
		\[
		{E^*}=\big\{\varepsilon_{\{(a, b)\}}: (a, b)\in\R^2, (a, b)\not=(0, 0)\big\},
		\] 
		of $(\R^\mathbb N)^*$, defined by 
		$
		\varepsilon_{\{(a, b)\}}(g_{(r, s)})=\begin{cases} 1, &\text{if $(a, b)=(r, s)$},\\
			0, & \text{otherwise}.
		\end{cases}
		$ We observe that  $E^*$ is a free set of  $(\R^\mathbb N)^*$.  Thus $E^*$ can be extended to a basis $\mathcal B^*=\big\{\varepsilon_A: A\in\mathcal P(\R^2)\big\}$ of $(\R^\mathbb N)^*$
		by Theorem~\ref{T: Existence of Basis}, where $\varepsilon_{A}=\varepsilon_{\{(a, b)\}}$ for all $A=\{(a, b)\}$ such that $(a, b)\not=(0,0)$. The subspace 
		$(\R^\mathbb N)^*_\mathcal B={\rm span}\big\{\varepsilon_{\{(a, b)\}}: (a, b)\in\R^2, (a, b)\not=(0, 0)\big\}$ of $(\R^\mathbb N)^*$ is the \emph{restricted dual} of $\R^\mathbb N$ relative to the base $\mathcal B$ (Definition~\ref{D: Embeddings & Restricted Duals}).
	\end{itemize}
\end{example}
\begin{example}[$\C^\mathbb N$ and its Dual]\label{Ex: CN and its Dual} Let $\C^\mathbb N$  denote the $\C$-vector space  of all sequences in $\C$ and let $(\C^\mathbb N)^*$  be its algebraic dual. Similarly to the previous  example, we have  $\dim(\C^\mathbb N)=\card\,(\C^\mathbb N)=\mathfrak{c}$ and  $\dim{(\C^\mathbb N)^*}=\card{(\C^\mathbb N)^*}=\mathfrak{c}_+$  Consequently,  $\C^\mathbb N\cong \C_0^{\R^2}$ and $(\C^\mathbb N)^*\cong \C_0^{\mathcal P(\R^2)}$.
\end{example}
\begin{example}[The Space of Polynomials and its Dual] \label{Ex: The Space of Polynomials and its Dual}  Let $\C[z]$ denote the $\C$-vector space consisting of all polynomials with coefficients in $\C$. Let $(\C[z])^*$ denote the algebraic dual of $\C[z]$.\begin{itemize}  
		\item We have $\dim(\C[z])=\aleph_0$, because $(1, z, z^2,\dots)$ is (obviously) a basis of $\C[z]$. Also, $\card\,\C[z]=\max\{\dim\C[z], \card\,\C\}=
		\max\{\aleph_0, \mathfrak{c}\}=\mathfrak{c}$ by (\ref{E: Cardinality}).
		
		\item Next, $\dim(\C[z])^*=\max\{(\aleph_0)_+, \card\,\C\}=\max\{\mathfrak{c}, \mathfrak{c}\}=\mathfrak{c}$ by (\ref{E: Dimension of Dual})
		and\linebreak $\card\,{(\C[z])^*}=\max\{\dim(\C[z])^*, \card\,\C\}=\max\{\mathfrak{c}, \mathfrak{c}\}=\mathfrak{c}$ by (\ref{E: Cardinality of Dual}). Thus,  $\dim\C[z]=\aleph_0$ and $\card\,\C[z]=\dim{(\C[z])^*}=\card\,{(\C[z])^*}=\mathfrak{c}$.
	\end{itemize}
\end{example}

The next example is important for the rest of the article.

\begin{example}[$\mathcal D(\Omega)$ and its Dual]\label{Ex: D(Omega) and its Dual}  Let $\Omega$ be an open set of $\R^d$ and $\mathcal D(\Omega)=\mathcal C_0^\infty(\Omega)$ denote the space of test-functions on $\Omega$ (Vladimirov~\cite{vVladimirov}).  Let  $\mathcal D^*(\Omega)$ denote the \emph{algebraic dual} of $\mathcal D(\Omega)$ (not to be confused with the space of Schwartz distributions $\mathcal D^\prime(\Omega)$, Vladimirov~\cite{vVladimirov}).
	\begin{itemize}
		\item  We show that $\dim\mathcal D(\Omega)=\card\,\mathcal D(\Omega)=\mathfrak{c}$. Indeed,  $\card\,\mathcal D(\Omega)\leq\mathfrak{c}$, because the mapping $\psi: \mathcal D(\Omega)\mapsto\C^{\Omega\cap\mathbb Q^d}$, $\psi(\varphi)=\varphi\rest\Omega\cap\mathbb Q^d$, is an injection (due to the continuity of $\varphi$) and $\card(\C^{\Omega\cap\mathbb Q^d})=2^{\aleph_0}=\mathfrak{c}$. Thus $\dim\mathcal D(\Omega)\leq\mathfrak{c}$ since  $\dim\mathcal D(\Omega)\leq\card\, \mathcal D(\Omega)$ holds trivially. To show that $\dim\mathcal D(\Omega)\geq\mathfrak{c}$, we observe that the set $E=\{\tau_h\varphi\in\mathcal D(\Omega): h\in\R^d\}$, is a \emph{free set} of $\mathcal D(\Omega)$ (Definition~\ref{D: Basis and Dimension}). Here $\varphi\in\mathcal D(\Omega)$ is a (fixed) \emph{non-zero test-function} and $(\tau_h\varphi)(x)=\varphi(x-h)$.  Indeed, suppose $\sum_{n=1}^m\, c_n \tau_{h_n}\varphi =0$ for some $m\in\mathbb N$, some $c_n\in\C$ and some mutually distinct $h_n\in\R^d$ such that  $\tau_{h_n}\varphi\in\mathcal D(\Omega)$. The Fourier transform produces $\big(\sum_{n=1}^m\, c_n e^{-ih_nz}\big)\mathcal F[\varphi](z) =0$, where both $\sum_{n=1}^m\, c_n e^{-ih_nz}$ and $\mathcal F[\varphi](z)$ are entire functions. So, we can cancel $\mathcal F[\varphi](z)$, because the ring of entire functions forms an integral domain. However, $\sum_{n=1}^m\, c_n e^{-ih_nz}=0$ implies $c_1=\dots =c_m=0$ (as desired), since the exponents are linearly independent.  Thus $\dim\mathcal D(\Omega)=\mathfrak{c}$. Finally,  $\card\,\mathcal D(\Omega)=\max\{\dim\mathcal D(\Omega), \card\,\C\}=\max\{\mathfrak{c}, \mathfrak{c}\}=\mathfrak{c}$ (as required) by (\ref{E: Cardinality}). Consequently,  $\mathcal D(\Omega)\cong\C_0^{\R^d\times\R}$, since $\card(\R^d\times\R)=\mathfrak{c}$ (Section~\ref{S: Coordinate Isomorphism}). Next, the free set $E$ can be extended to a basis $\mathcal B_{\mathcal D(\Omega)}$ of $\mathcal D(\Omega)$ of the form $\mathcal B_{\mathcal D(\Omega)}=\{\varphi_{h,r}: (h, r)\in\R^d\times\R\}$, where $\varphi_{h,0}=\tau_h\varphi$ for all $h\in\R^d$ such that $\tau_h\varphi\in\mathcal D(\Omega)$ (Theorem~\ref{T: Existence of Basis}).
		\item We show that $\dim\mathcal D^*(\Omega)=\card\,\mathcal D^*(\Omega)=\mathfrak{c}_+$. Indeed, $\dim\mathcal D^*(\Omega)
		=$\linebreak $\max\{(\dim\mathcal D(\Omega))_+, \card\,\C\}=\max\{\mathfrak{c}_+, \mathfrak{c}\}=\mathfrak{c}_+$ by (\ref{E: Dimension of Dual}). Also, we observe that $\card\,\mathcal D^*(\Omega)= \mathfrak{c}_+$ by Corollary~\ref{C: Two More Equalities}, since $\card\,\C<\dim\mathcal D^*(\Omega)$. Thus $\mathcal D^*(\Omega)\cong\C_0^{\mathcal P(\R^d\times\R)}$, since $\card\,\mathcal P(\R^d\times\R)=\mathfrak{c}_+$ (Section~\ref{S: Coordinate Isomorphism}).
		\item Here is one particular choice of a basis $\mathcal B_{\mathcal D^*(\Omega)}$ of $\mathcal D^*(\Omega)$: Let $\{\Phi_{g,s}: (g, s)\in\R^d\times\R\}$ be the subset of $\mathcal D^*(\Omega)$ such that $\Phi_{g,s}(\varphi_{h,r})=\begin{cases}1, &\text{if $(g, s)=(h, r)$;}\\0, &\text{if $(g, s)\not=(h, r)$.}\end{cases}$ It is clear that $\{\Phi_{g,s}: (g, s)\in\R^d\times\R\}$ is a free set of $\mathcal D^*(\Omega)$ and thus, it can be extended to a basis $\mathcal B_{\mathcal D^*(\Omega)}=\{\Phi_X: X\in\mathcal P(\R^d\times\R)\}$ of $\mathcal D^*(\Omega)$ by Theorem~\ref{T: Existence of Basis}, where $\Phi_{\{(h,r)\}}=\Phi_{h,r}$ for all $(h, r)\in\R^d\times\R$. The subspace $\mathcal D^*(\Omega)_{\mathcal B_{\mathcal D(\Omega)}}={\rm span}\{\Phi_{h,r}: (h, r)\in\R^d\times\R\}$ of  $\mathcal D^*(\Omega)$ is the \emph{restricted dual} of  $\mathcal D(\Omega)$ relative to the base $\mathcal B_{\mathcal D(\Omega)}$ (Definition~\ref{D: Embeddings & Restricted Duals}).
	\end{itemize}
\end{example}

\begin{example}[$\mathcal E(\Omega)$ and its Dual]\label{Ex: E(Omega) and its Dual}  Let $\mathcal E(\Omega)=\mathcal C^\infty(\Omega)$
	and let  $\mathcal E^*(\Omega)$ denote the algebraic dual of $\mathcal E(\Omega)$. Similarly to the previous example, we have $\card\,\mathcal E(\Omega)=\dim\mathcal E(\Omega)=\mathfrak{c}$ and $\card\,\mathcal E^*(\Omega)=\dim\mathcal E^*(\Omega)=\mathfrak{c}_+$. Thus $\mathcal E(\Omega)\cong \C_0^{\R^d\times\R}$ and $\mathcal E^*(\Omega)\cong \C_0^{\mathcal P(\R^d\times\R)}$ (Section~\ref{S: Coordinate Isomorphism}). Notice that $\mathcal D^*(\Omega)\subset\mathcal E^*(\Omega)$ by Corollary~\ref{C: Embedding of Duals} \emph{in sharp contrast to} the embedding of the \emph{distributions with compact support} $\mathcal E^\prime(\Omega)\subset\mathcal D^\prime(\Omega)$ in distribution theory (Vladimirov~\cite{vVladimirov}, p. 43). 
\end{example}
\begin{example} [$\mathcal D^\prime(\Omega)$ and its Dual] \label{Ex: D'(Omega) and its Dual} Let $\mathcal D^\prime(\Omega)$ denote the space of \emph{Schwartz distributions}  on $\Omega$ (Vladimirov~\cite{vVladimirov})) and let $\mathcal D^{\prime*}(\Omega)$ denote the \emph{algebraic dual} of $\mathcal D^\prime(\Omega)$.
	\begin{itemize}
		\item  We have $\card\,\mathcal D^\prime(\Omega)=\dim\mathcal D^\prime(\Omega)=\mathfrak{c}$, because $\mathcal D^\prime(\Omega)$ is sequentially separable in the weak-star-topology (the topology of the pointwise convergence, Definition~\ref{D: The Space D*(Omega)}) and also, $\mathcal D^\prime(\Omega)$ is a countable union of weak-star-bounded sets of $\mathcal D^\prime(\Omega)$; the polars of the elements of a countable subsets of $\mathcal D(\Omega)$.  It remains to take into account that the weak-star-bounded sets of $\mathcal D^\prime(\Omega)$ are metrizable in the weak-star-topology (K\"{o}the~\cite{gKotheI}, p. 261, \S 21.3 (4)). Thus, $\mathcal D^\prime(\Omega)$ is of cardinality at most $\mathfrak{c}$. Consequently, $\mathcal D^\prime(\Omega)$ and $\C_0^{\R^d\times\R}$ are isomorphic $\C$-vector spaces, since $\card(\R^d\times\R)=\mathfrak{c}$.
		
		\item Next, we have $\card\,\mathcal D^{\prime*}(\Omega)=\dim\mathcal D^{\prime*}(\Omega)=\mathfrak{c}_+$. Indeed,  $\dim\mathcal D^{\prime*}(\Omega)=\max\{\mathfrak{c}_+, \card\,\C\}=\mathfrak{c}_+$ by (\ref{E: Dimension of Dual}) and $\card\,\mathcal D^{\prime*}(\Omega)=\max\{\mathfrak{c}_+, \mathfrak{c}\}=\mathfrak{c}_+$  by (\ref{E: Cardinality of Dual}). Consequently, $\mathcal D^{\prime*}(\Omega)$ and $\C_0^{\mathcal P(\R^d\times\R)}$ are isomorphic $\C$-vector spaces (Section~\ref{S: Coordinate Isomorphism}).
	\end{itemize}
\end{example}
\begin{example}[Hilbert Space $\mathcal H$ and its Dual]\label{Ex: Hilbert Space H and its Dual} Let $\mathcal H$ be a separable (infinite-dimensional) Hilbert space over $\C$ with an inner product $\langle\cdot, \cdot\rangle$. Let $(e_1, e_2,\dots)$ be a Hilbert (not Hamel) basis of $\mathcal H$. Recall that $H={\rm span}\{e_1, e_2,\dots\}$ is a proper inner subspace of $\mathcal H$, which is dense in $\mathcal H$ relative to the norm topology generated by the norm $\sqrt{\langle\cdot, \cdot\rangle}$. Let  $\mathcal H^*$ denote the \emph{algebraic dual} of $\mathcal H$ and  $\mathcal H^{**}$ denote the \emph{algebraic double dual} of $\mathcal H$.
	\begin{itemize}
		\item We have $\card\,\mathcal H=\dim\mathcal H=\mathfrak{c}$ and thus $\mathcal H\cong \C_0^\R$. Indeed,  $\card\,\mathcal H\geq\mathfrak{c}$ by (\ref{E: Cardinality}), since $\card\,\C=\mathfrak{c}$. The inequality $\card\,\mathcal H\leq\mathfrak{c}$ follows from fact that $\mathcal H$ is a metric space (hence, of cardinality at most $\mathfrak{c}$). Thus, $\card\,\mathcal H=\mathfrak{c}$. Next, $\dim\mathcal H\leq\card\,\mathcal H= \mathfrak{c}$ holds trivially.  To show $\dim\mathcal H\geq \mathfrak{c}$, observe that $\dim(\mathcal H)=\dim(\mathcal L^2(\R))$, because $\mathcal H$ and $\mathcal L^2(\R)$ are vector-isomorphic (as separable Hilbert spaces), $\mathcal D(\R)\subset\mathcal L^2(\R)$ and $\dim\mathcal D(\R)=\mathfrak{c}$ (Example~\ref{Ex: D(Omega) and its Dual}). Thus, $\dim\mathcal H=\mathfrak{c}$ as required.
		\item How to choose a Hamel basis of $\mathcal H$?  Unfortunately, Hilbert spaces (separable or not) do not have orthonormal Hamel bases. Suppose (seeking a contradiction) that $\mathcal B$ is an orthonormal Hamel basis of a Hilbert space $\mathcal H$. Let $(e_1, e_2,\dots)$ be an orthonormal sequences in $\mathcal H$. Consider the sequence $(v_1, v_2,\dots)$ in $\mathcal H$ by $v_1=e_1, v_2=e_1+\frac{1}{2}e_2$, etc. $v_n=\sum_{k=1}^n \frac{e_k}{k}$. Then $(v_1, v_2,\dots)$ is a Cauchy sequence, but it is divergent. Indeed, suppose (seeking a contradiction again) that $\lim_{n\mapsto\infty} ||v_n-v||$ for some $v\in\mathcal H$. We have $v=\sum_{k=1}^mc_kw_k$ for some $m\in\mathbb N$, some $c_k\in\C$ and some $w_k\in\mathcal B$. After replacing, we get $\lim_{n\mapsto\infty} ||\sum_{k=1}^n \frac{e_k}{k}-\sum_{k=1}^mc_kw_k||=0$, a contradiction, since in $(w_1,\cdots,w_m, e_1, e_2,\dots)$ there are not more than finitely many repetitions and the set $\{w_1,\cdots,w_m, e_1, e_2,\dots\}$ consisting of mutually orthogonal unit vectors only. Thus, $\mathcal H$ is non-complete, another contradiction.
		
		\item So, we have to choose a non-orthonormal Hamel basis of $\mathcal H$, which is desirable to be as much close to an orthonormal as possible. One way to do this is to start from a Hilbert (non-Hamel) orthonormal basis $(e_1, e_2,\dots)$ of $\mathcal H$ (mentioned already above) and to extend it (non-uniquely) to a Hamel basis  $\mathcal B_\mathcal H=\{e_r: r\in\R\}$ of $\mathcal H$, by Theorem~\ref{T: Existence of Basis}. Note that the basis $\mathcal B_\mathcal H$ is non-orthogonal (although it is an extension of an orthonormal Hilbert basis). 
		\item We show now that $\card\,{\mathcal H^*}=\dim{\mathcal H^*}=\mathfrak{c}_+$ and thus $\mathcal H^*\cong \C_0^{\mathcal P(\R)}$ (Section~\ref{S: Coordinate Isomorphism}). Indeed, we have $\dim(\mathcal H^*)=\max\{(\dim V)_+, \card\,\C\}=\max\{\mathfrak{c}_+, \mathfrak{c}\}=\mathfrak{c}_+$ by (\ref{E: Dimension of Dual}). Also, $\card(\mathcal H^*)=\max\{\dim\mathcal H^*, \card\,\C\}=\max\{\mathfrak{c}_+, \mathfrak{c}\}=\mathfrak{c}_+$ by (\ref{E: Cardinality of Dual}).  Let $\{\varepsilon_s: s\in\R\}$ be the subset of $\mathcal H^*$ defined by $\varepsilon_s(e_r)=\delta_{sr}$ for all $r, s\in\R$. It is clear that $\{\varepsilon_s: s\in\R\}$ is a free set of $\mathcal H^*$ and thus  it can be extended to a basis $\mathcal B_{\mathcal H^*}=\{\varepsilon_s: s\in\mathcal P(\R)\}$ of $\mathcal H^*$. The subspace  $\mathcal H^*_{\mathcal B_\mathcal H}={\rm span}\{\varepsilon_s: s\in\R\}$ of $\mathcal H^*$ is the \emph{restricted dual} of $\mathcal H$ relative to $\mathcal B_\mathcal H$ (Definition~\ref{D: Embeddings & Restricted Duals}).
		
		\item Finally, we have $\dim{\mathcal H^{**}}$$=\card\,{\mathcal H^{**}}=(\mathfrak{c}_+)_+$ by (\ref{E: Cardinality})-(\ref{E: Cardinality of Dual}) and thus $\mathcal H^{**}\cong \C_0^{\mathcal P(\mathcal P(\R))}$ (Section~\ref{S: Coordinate Isomorphism}). Let $\iota: \mathcal H\mapsto\mathcal H^{**}$ be the \emph{canonical embedding} of $\mathcal H$ into $\mathcal H^{**}$, defined by $\iota(\varphi)(T)=T(\varphi)$ for all $T\in\mathcal H^*$ (Definition~\ref{D: Embeddings & Restricted Duals}). Then  $\iota[\mathcal B_\mathcal H]$ is a Hamel basis of $\iota[\mathcal H]$. Let $W$ be an algebraic complement of $\iota[\mathcal H]$ to $\mathcal H^{**}$, i.e.  $\mathcal H^{**}=\iota[\mathcal H]\oplus W$(Corollary~\ref{C: Algebraic Complement}). Let $\mathcal B_W=\{w_A: A\in\mathcal P(\mathcal P(\R))\setminus\R\}$ be a Hamel basis of $W$. Then $\mathcal B^{**}=\iota[\mathcal B_\mathcal H]\cup\mathcal B_W$ is a Hamel basis of $\mathcal H^{**}$. 
	\end{itemize}
\end{example}
\begin{remark}[An Alternative]\label{R: An Alternative} Alternatively to the example above, we can define a (new) inner product on $\mathcal H$ by $(v, w)=\sum_{r\in\spm(v)\,\cap\,\spm(w)}\bar{a}_rb_r$, where $v=\sum_{r\in\spm(v)}a_re_r\in\mathcal H$ and $ w=\sum_{s\in\spm(w)}b_se_s\in \mathcal H$ (Definition~\ref{D: Spectrum}). We observe that $\spm(e_r)=\{r\}$. Thus, $(e_r, e_s)=\delta_{rs}$ for all $r, s\in\R$. The latter means that $\mathcal B_\mathcal H=\{e_r: r\in\R\}$ \emph{is an orthonormal Hamel basis of} $\mathcal H$ relative to $(\cdot, \cdot)$. We observe that $(\cdot, \cdot)$ coincides with 
	$\langle\cdot, \cdot\rangle$ on $H={\rm span}\{e_1, e_2,\dots\}$, since $\mathcal B_\mathcal H=\{e_r: r\in\R\}$ is an extension of $\{e_1, e_2,\dots\}$. Notice however, that $\big(\mathcal H, (\cdot, \cdot)\big)$ is not a Hilbert space by what was explained above; it is non-complete relative to the norm $\sqrt{(\cdot,\cdot)}$. Rather, $\big(\mathcal H, (\cdot, \cdot)\big)$ is merely an infinite-dimensional inner vector space over $\C$, which admits an orthonormal Hamel basis and which shares with $\big(\mathcal H, \langle\cdot, \cdot\rangle\big)$ a common inner subspace $H$.
\end{remark}
Summarizing, the vector spaces $\mathcal D(\Omega)$, $\mathcal E(\Omega)$, $\mathcal D^\prime(\Omega)$, $\mathcal E^\prime(\Omega)$, $\mathcal L^2(\Omega)$, $l_2(\C)$, $\mathcal H$, $\C_0^{\R^d\times\R}$ and $\C_0^\R$ are mutually isomorphic and all of dimension $\mathfrak{c}$. Also,  $\mathcal D^*(\Omega)$, $\mathcal D^{\prime *}(\Omega)$, $\mathcal E^*(\Omega)$, $(\mathcal L^2(\Omega))^*$, $(l_2(\C))^*$, $\mathcal H^*$, $\C_0^{\mathcal P(\R^d\times\R)}$ and $\C_0^{\mathcal P(\R)}$ are mutually isomorphic and all of dimension $\mathfrak{c}_+$. Similarly, $\mathcal D^{**}(\Omega)$, $\mathcal D^{\prime **}(\Omega)$, $\mathcal E^{**}(\Omega)$, $(\mathcal L^2(\Omega))^{**}$, $(l_2(\C))^{**}$, $\mathcal H^{**}$ and $\C_0^{\mathcal P(\mathcal P(\R))}$ are mutually isomorphic and all of dimension $(\mathfrak{c}_+)_+$.

The next several examples are about vector spaces over a field of \emph{non-standard complex numbers} $^*\C$. The reader who is unfamiliar with non-standard analysis (Robinson~\cite{aRob66}), might skip these examples and resume the reading from Section~\ref{S: Linear Functionals in D*(Omega) as Generalized Functions: The Main Result}. Recall that for every infinite cardinal $\kappa$ there exists a unique (up to a field isomorphism) non-standard extension $^*\C$ of $\C$ in a $\kappa_+$-saturated ultra-power non-standard model with the set of individuals $\R$ (Chang \& Keisler~\cite{ChKeisl}; for a presentation we refer also to the Appendix in Lindstr\o m~\cite{tLindstrom}).  Recall as well that $^*\C$ is an algebraically closed non-Archimedean field containing $\C$ (as a subfield) with $\card(^*\C)=\kappa_+$. Also, $^*\C={^*\R}(i)$, where $^*\R$ is a $\kappa_+$-saturated real closed non-Archimedean field containing $\R$ as a subfield. We should alert the reader that the asterisk in front, ${^*\C}$, has nothing to do with the asterisk after, ${\C}^*$. 

\begin{example}[The Space $^*\C|\C$ and its Dual]\label{Ex: The Space *C|C and its Dual} Let $^*\C\,|\,\C$ denote the $\C$-vector space of $^*\C$ and $(^*\C\,|\,\C)^*$ stand for its dual. We have $\card(^*\C)=\card(^*\C|\C)=\dim(^*\C|\C)=\kappa_+$ and $\card(^*\C|\C)^*=\dim(^*\C|\C)^*=(\kappa_+)_+$. Indeed,
	\begin{itemize}
		\item If $\kappa>\aleph_0$, then
		$\dim(^*\C|\C)=\kappa_+$ by Corollary~\ref{C: Two Equalities}, since $\card\,\C<\card(^*\C)$. Next,  ${\dim(^*\C|\C)^*}=\max\{(\kappa_+)_+, \mathfrak{c}\}=(\kappa_+)_+$ by (\ref{E: Dimension of Dual}) and  ${\card(^*\C|\C)^*}=(\kappa_+)_+$ by Corollary~\ref{C: Two More Equalities}, since $\card\,\C<\dim(^*\C|\C)^*$.
		
		\item If $\kappa=\aleph_0$, then $\dim(^*\C|\C)\leq\mathfrak{c}$ holds trivially since $\card(^*\C)=(\aleph_0)_+=\mathfrak{c}$. To show that $\dim(^*\C|\C)=\mathfrak{c}$, we observe that the set $E=\{\rho^r: r\in\R\}$ is a free subset of $^*\C$ of cardinality $\mathfrak{c}$. Here  $\rho$ is a (fixed) positive infinitesimal in $^*\C$ (actually, $\rho\in{^*\R}$). The next two formulas,  $\dim(^*\C|\C)^*=\card(^*\C|\C)^*=\mathfrak{c}_+$, follow by exactly the same arguments as in case $\kappa>\aleph_0$.
	\end{itemize}
\end{example}

\begin{example}  [The Space of Non-Standard Polynomials and its Dual] \label{Ex: The Space of Non-Standard Polynomials and its Dual} Let ${^*\C}[z]$ denote the vector space over the field $^*\C$ consisting of all polynomials in one variable with coefficients in $^*\C$. Let $(^*\C[z])^*$ stand for the \emph{dual space} of ${^*\C}[z]$ (warning again, the two asterisks have different meaning). As in Example~\ref{Ex: The Space of Polynomials and its Dual},  $\dim({^*\C}[z])= \aleph_0$, because $\{1, z, z^2,\dots\}$ is (obviously) a Hamel basis of ${^*\C}[z]$, and  $\card(^*\C[z])=\dim(^*\C[z])^*=\card(^*\C[z])^*=
	\kappa_+$. 
\end{example}

Let $\kappa$ be an infinite cardinal and let $^*\mathcal E(\Omega)$ be the non-standard extension of $\mathcal E(\Omega)=\mathcal C^\infty(\Omega)$ in a $\kappa_+$-saturated ultra-power non-standard model with set of individuals $\R$ (the same non-standard framework as in Examples~\ref{Ex: The Space *C|C and its Dual}). We have $\card(^*\mathcal E(\Omega))=\kappa_+$. Indeed, $\card(^*\mathcal E(\Omega))\leq\kappa_+$ holds, because $^*\mathcal E(\Omega)=\mathcal E(\Omega)^I/\mathcal U$, where $\card(I)=\kappa$ and  $\mathcal U$ is a $\kappa_+$-good free ultrafilter on $I$ (Chang \& Keisler~\cite{ChKeisl} or Lindstr\o m~\cite{tLindstrom}). Thus, $\card(\mathcal E(\Omega)^I)=\mathfrak{c}^\kappa=\kappa_+$. On the other hand, $\card(^*\mathcal E(\Omega))\geq\kappa_+$ holds, because  $^*\mathcal E(\Omega)$ is $\kappa_+$-saturated. Thus, $\card(^*\mathcal E(\Omega))=\kappa_+$. With this in mind, we have the following example.
\begin{example}[The Space $^*\mathcal E(\Omega)|\C$ and its Dual] \label{Ex: The Space *E(Omega)|C and its Dual} 
	
	Let  $^*\mathcal E(\Omega)|\C$ denote the vector space of $^*\mathcal E(\Omega)$ over $\C$ and   $(^*\mathcal E(\Omega)|\C)^*$ be the algebraic dual of  $^*\mathcal E(\Omega)|\C$ (the asterisks in front and after  $\mathcal E(\Omega)$ have different meaning). We have $\card(^*\mathcal E(\Omega))=\card(^*\mathcal E(\Omega)|\C)=\dim(^*\mathcal E(\Omega)|\C)=\kappa_+$ and $\card{(^*\mathcal E(\Omega)|\C)^*}=\dim{(^*\mathcal E(\Omega)|\C)^*}=(\kappa_+)_+$.  Indeed,
	
	\begin{itemize}
		\item If $\kappa>\aleph_0$, then
		$\dim(^*\mathcal E(\Omega)|\C)=\kappa_+$ by Corollary~\ref{C: Two Equalities}, since $\card\,\C<\card(^*\mathcal E(\Omega))$. Next,  ${\dim(^*\mathcal E(\Omega)|\C)^*}=\max\{(\kappa_+)_+, \mathfrak{c}\}=(\kappa_+)_+$ by (\ref{E: Dimension of Dual}) and  ${\card(^*\mathcal E(\Omega)|\C)^*}=(\kappa_+)_+$ by Corollary~\ref{C: Two More Equalities}, since $\card\,\C<\dim(^*\mathcal E(\Omega)|\C)^*$.
		
		\item If $\kappa=\aleph_0$, then $\dim(^*\mathcal E(\Omega)|\C)\leq\mathfrak{c}$ holds trivially since $\card(^*\mathcal E(\Omega))=(\aleph_0)_+=\mathfrak{c}$. To show that $\dim(^*\mathcal E(\Omega)|\C)=\mathfrak{c}$, we observe that the set $E=\{e^{\lambda x}: \lambda\in\C\}$ is a free subset of $^*\mathcal E(\Omega)$ of cardinality $\mathfrak{c}$. The next two formulas,  $\dim(^*\mathcal E(\Omega)|\C)^*=\card(^*\mathcal E(\Omega)|\C)^*=\mathfrak{c}_+$, follow by exactly the same arguments as in case $\kappa>\aleph_0$.
	\end{itemize}
\end{example}

\begin{example}[The Space $^*\mathcal E(\Omega)|^*\C$ and its Dual] \label{Ex: The Space *E(Omega)|*C and its Dual}  Let $^*\mathcal E(\Omega)|^*\C$ denote the vector space of $^*\mathcal E(\Omega)$ over the field $^*\C$ and $(^*\mathcal E(\Omega)|^*\C)^*$ be its dual space. As in the previous example, we have $\card(^*\mathcal E(\Omega))=\card(^*\mathcal E(\Omega)|^*\C)=\dim(^*\mathcal E(\Omega)|^*\C)=\kappa_+$ and $\card{(^*\mathcal E(\Omega)|^*\C)^*}=\dim{(^*\mathcal E(\Omega)|^*\C)^*}=(\kappa_+)_+$. Indeed,  $\dim(^*\mathcal E(\Omega)|^*\C)\leq\kappa_+$ holds trivially, because $\card(^*\mathcal E(\Omega))=\kappa_+$ (Example~\ref{Ex: The Space *E(Omega)|C and its Dual}). To show that $\dim(^*\mathcal E(\Omega)|^*\C)=\kappa_+$, we observe that $E=\{e^{\lambda x}: \lambda\in{^*\C}\}$ is a free subset of $^*\mathcal E(\Omega)$ of cardinality $\kappa_+$. Next, $\dim{(^*\mathcal E(\Omega)|^*\C)^*}=\max\{(\kappa_+)_+, \kappa_+\}=(\kappa_+)_+$ by (\ref{E: Dimension of Dual}). Finally,  $\card{(^*\mathcal E(\Omega)|^*\C)^*}=(\kappa_+)_+$ by Corollary~\ref{C: Two More Equalities}, since $\card(^*\C)<\dim(^*\mathcal E(\Omega)|\C)^*$.
\end{example}
\section{Linear Functionals in $\mathcal D^*(\Omega)$ as Generalized Functions: The Main Result}  \label{S: Linear Functionals in D*(Omega) as Generalized Functions: The Main Result}

We supply $\mathcal D^*(\Omega)$ (Example~\ref{Ex: D(Omega) and its Dual}) with the structure of a sheaf of differential $\C$-vector spaces (and, more generally, a sheaf of differential modules over $\mathcal E(\Omega)=\mathcal C^\infty(\Omega)$). This structure is inherited from $\mathcal D(\Omega)$ by duality. Our framework is the infinite-dimensional linear algebra presented in Sections~\ref{S: Infinite-Dimensional Linear Algebra}-\ref{S: Linear Maps and Operators} applied to particular case $V=\mathcal D(\Omega)$ and its dual ${V^*}=\mathcal D^*(\Omega)$. As before, $\Omega$ stands for a (generic) open subset of $\R^d$ (Section~\ref{S: Notations}).

We are trying to convince the reader that $\mathcal D^*(\Omega)$ deserves to be treated as a space of generalized functions. In what follows the space $\mathcal L_{loc}(\Omega)$ (and more literarily, $S_\Omega[\mathcal L_{loc}(\Omega)]$ explained below) presents the \emph{set of classical functions} as apposed to the \emph{generalized functions} in $\mathcal D^*(\Omega)$. The embedding of Schwartz distributions (Vladimirov~\cite{vVladimirov}) in  $\mathcal D^*(\Omega)$ will be discussed in the next section.

Our approach is a refinement and generalization of distribution theory. That is why, starting from this section, we follow the tradition of distribution theory and use the \emph{bracket notation} mentioned in (Remark~\ref{R: Bracket Notation}): We shall write  $\langle T, \varphi\rangle$ instead of $T(\varphi)$ for the evaluation of $T\in \mathcal D^*(\Omega)$ at $\varphi\in \mathcal D(\Omega)$.

\begin{definition}[The Space $\mathcal D^*(\Omega)$]\label{D: The Space D*(Omega)}
	\begin{D-enum}
		
		\item\label{Item: Product} Let $f\in\mathcal E(\Omega)$ and $T\in\mathcal D^*(\Omega)$. We define the product $f\,T\in\mathcal D^*(\Omega)$ by $\langle f\,T,\varphi\rangle=\langle T, f\varphi\rangle$ for all $\varphi\in\mathcal D(\Omega)$.
		\item Let $X$ be an open set of $\R^d$ such that $X\subseteq\Omega$ and $T\in\mathcal D^*(\Omega)$. We define the \emph{restriction} $T\rest X\in\mathcal D^*(X)$ of $T$ on $X$ by $\langle T\!\rest \!X, \varphi\rangle=\langle T, \bar{\varphi}\rangle$ for all $\varphi\in\mathcal D(X)$, where $\bar{\varphi}$ is the extension of $\varphi$ from $X$ to $\Omega$ by zero-values on $\Omega\setminus X$. 
		
		\item Let $T_n, T\in\mathcal D^*(\Omega)$, $n\in\mathbb N$. We define the \emph{weak (pointwise) convergence} $T_n\overset{*}\mapsto T$ in $\mathcal D^*(\Omega)$ by $\lim_{n\mapsto\infty}\bra T_n,\varphi\ket=\langle T, \varphi\rangle $ for all $\varphi\in\mathcal D(\Omega)$, where $\lim_{n\mapsto\infty}$ stands for the usual limit in $\C$.
		
		\item\label{Item: Differentiation} Let $\alpha= (\alpha_1,\dots,\alpha_d)\in\mathbb N_0^d$ be a multi-index and $|\alpha|=\alpha_1+\dots+\alpha_d$. Let $\partial^\alpha=\frac{\partial^{|\alpha|}}{\partial x_1^{\alpha_1}\dots\partial x_d^{\alpha_d}}$ be the usual partial derivative operator. We define $\partial^\alpha: \mathcal D^*(\Omega)\mapsto\mathcal D^*(\Omega)$ by $\langle \partial^\alpha T, \varphi\rangle=(-1)^{|\alpha|}\bra T, \partial^\alpha\varphi\ket$ for all $\varphi\in\mathcal D(\Omega)$.
		
		\item  We define the \emph{Schwartz embedding} $S_\Omega: \mathcal L_{loc}(\Omega)\mapsto \mathcal D^*(\Omega)$ by $S_\Omega(f)=T_f$, where $\langle T_f, \varphi\rangle=\int_\Omega f(x)\varphi(x)\, dx$ for all $\varphi\in\mathcal D(\Omega)$. 
		\item  Let $X$ and  $Y$ be two open sets of $\R^d$, $\theta\in\Diff(X, Y)$ be a \emph{diffeomorphism} from $X$ to $Y$ and $J_\theta: X\to Y$,  $J_\theta=\big|\det\big(\frac{\partial\theta}{\partial x}\big)\big|$, be the corresponding \emph{Jacobian determinant}.  We define  the \emph{change of variables} $\theta_*: \mathcal{D}^*(X)\to \mathcal{D}^*(Y)$ by the formula
		$\langle\theta_*(T)(y), \varphi(y)\rangle=\langle T(x),\, (\varphi\circ\theta)(x)J_\theta(x)\rangle$ for all $\varphi\in\mathcal D(Y)$. We sometimes write $T(\theta^{-1})$ instead of $\theta_*(T)$.
	\end{D-enum}
\end{definition}
\begin{examples}[Inflections and Translations]\label{Ex: Inflections and Translations}
	\begin{D-enum}
		\item Let $X=Y=\R^d$ and $\theta(x)=-x$. We denote the \emph{inflection} $\theta_*(T)$ of $T$ by $\Check{T}$, i.e. $\langle \Check{T}, \varphi\rangle=\langle T, \Check{\varphi}\rangle$ for all $\varphi\in\mathcal D(\R^d)$, where $\Check{\varphi}(x)=\varphi(-x)$.
		\item Let (as above) $X=Y=\R^d$,  $h\in\R^d$ and let $\theta(x)=x+h$. We have $\langle(\theta_*T)(y), \varphi(y)\rangle=\langle T(x), \varphi(x+h)\rangle$ for all $\varphi\in\mathcal D(\R^d)$. It is customary to call the mapping $\tau_h: \mathcal D^*(\R^d)\mapsto\mathcal D^*(\R^d)$, $\tau_hT=\theta_*(T)$, \emph{translations} and often to write $T(y-h)$ instead of $(\theta_*T)(y)$.
	\end{D-enum}
\end{examples}
\begin{theorem}[Properties of $\mathcal D^*(\Omega)$]\label{T: Properties of D*(Omega)}
	\begin{T-enum}
		
		\item $\mathcal D^*(\Omega)$ is a \textbf{differential module} over $\mathcal E(\Omega)$. Consequently, $\mathcal D^*(\Omega)$ is a \textbf{vector space} over $\C$ and $S_\Omega[\mathcal L_{loc}(\Omega)]$ is a \textbf{vector subspace} of $\mathcal D^*(\Omega)$.
		
		\item The family $\{\mathcal D^*(\Omega)\}_{\Omega\in\mathcal T^d}$ is a \textbf{sheaf of differential modules} over $\mathcal E(\Omega)$ relative to the usual topology $\mathcal T^d$ on $\R^d$ and the restriction $\rest$ (Kaneko~\cite{aKan88}, p.16). Consequently, the family $\{\mathcal D^*(\Omega)\}_{\Omega\in\mathcal T^d}$ is a \textbf{sheaf of differential vector spaces over $\C$}. (Thus the \textbf{support}, $\supp(T)$, is well-defined for every $T\in\mathcal D^*(\Omega)$).
		
		\item The family $\{S_\Omega[\mathcal L_{loc}(\Omega)]\}_{\Omega\in\mathcal T^d}$ is a \textbf{subsheaf} of $\{\mathcal D^*(\Omega)\}_{\Omega\in\mathcal T^d}$ of both $\mathcal E(\Omega)$-modules and $\C$-vector spaces.  Thus $\supp(f)=\supp(S_\Omega(f))$ for every $f\in\mathcal L_{loc}(\Omega)$.
		
		\item \label{Item: Regular} The linear partial differential operator $P^*(x, \partial)\!\!: \mathcal D^*(\Omega)\mapsto\mathcal D^*(\Omega)$, $
		P^*(x, \partial)=\sum_{|\alpha|\leq m}c_\alpha(x)\partial^\alpha$, with $\mathcal C^\infty$-coefficients $c_\alpha\in\mathcal E(\Omega)$, is the \textbf{dual} of the operator $P(x, \partial)\!: \mathcal D(\Omega)\mapsto\mathcal D(\Omega)$, defined by $P(x, \partial)\varphi(x)=
		\sum_{|\alpha|\leq m}(-1)^{|\alpha|} \partial^\alpha\big(c_\alpha(x)\varphi(x)\big)$  (Definition~\ref{D: Duals and Co-Duals}). Consequently, $P^*(x, \partial)$ is \textbf{regular} (Definition~\ref{D: Regular Operators}) - thus \textbf{surjective} - \ifff the operator $P(x, \partial)$ is \textbf{injective}.
	\end{T-enum}
\end{theorem}
\begin{proof}
	\begin{Pf-enum}
		\item follows from the fact that  $f\in\mathcal E(\Omega)$ and $\varphi\in\mathcal D(\Omega)$  implies $f\varphi\in\mathcal D(\Omega)$.
		
		\item For a short proof we refer to  (Oberguggenberger~\cite{MO13}, p. 10).
		
		\item follows directly  from (i). 
		
		\item We calculate  $\langle P^*(x, \partial)T, \varphi\rangle=\sum_{|\alpha|\leq m}\langle c_\alpha\partial^\alpha T, \varphi\rangle=
		\sum_{|\alpha|\leq m}\langle\partial^\alpha T, c_\alpha\varphi\rangle
		\newline=\sum_{|\alpha|\leq m}(-1)^{|\alpha|}\langle T, \partial^\alpha(c_\alpha\varphi)\rangle=
		\langle T,\, \sum_{|\alpha|\leq m}(-1)^{|\alpha|} \partial^\alpha(c_\alpha\varphi)\rangle=
\langle T, P(x, \partial)\varphi\rangle$ for all $T\in\mathcal D^*(\Omega)$ and all $\varphi\in\mathcal D(\Omega)$. Thus $P^*(x, \partial)$ is the dual of $P(x, \partial)$. The rest follows immediately from Theorem~\ref{T: Surjective Dual} applied for $V=\mathcal D(\Omega)$ and $\mathcal O=P(x, \partial)$.
	\end{Pf-enum}
\end{proof}
\begin{theorem}[The Main Result] \label{T: The Main Result} Let $P^*(x, \partial)=\sum_{|\alpha|\leq m}c_\alpha(x)\partial^\alpha$ be a regular linear partial differential operator  with $\mathcal C^\infty$-coefficients. Then the equation $P^*(x, \partial)U=T$ is \textbf{solvable} in $\mathcal D^*(\Omega)$ in the sense that for any choice of $T$ in $\mathcal D^*(\Omega)$ this equation has a solution $U$ in $\mathcal D^*(\Omega)$. 
\end{theorem}
\begin{description}  
	\item[$\mathnormal{Proof\, 1}$]  An immediate consequence of Corollary~\ref{C: Solvability} applied for $V=\mathcal D(\Omega)$, ${V^*}=\mathcal D^*(\Omega)$, $\mathcal O= P(x, \partial)$ and ${\mathcal O^*}= P^*(x, \partial)$, since the operator $P^*(x, \partial)$ is the dual of $P(x, \partial)$ by Theorem~\ref{T: Properties of D*(Omega)}, \#~\ref{Item: Regular}.
	
	\item[$\mathnormal{Proof\, 2}$]   Here is an independent proof based on Theorem~\ref{T: Existence of Basis} only: Let $\ran (P(x, \partial))\subseteq\mathcal D(\Omega)$ denote the range of $P(x, \partial)$. We define the linear functional \linebreak
	$\Phi\!\!: \ran(P(x, \partial))\mapsto\C$ by $\bra \Phi, {P(x, \partial)}\varphi\ket=\bra T, \varphi\ket$ for all $\varphi\in\mathcal D(\Omega)$. The mapping $\Phi$ is well defined, since  ${P(x, \partial)}$ is injective by assumption, and  $\Phi\in(\ran(P(x, \partial)))^*$, because $T$ is linear. Let $\mathcal B_1$ be a (Hamel) basis of $ \ran(P(x, \partial))$. Then  $\mathcal B_1$ can be extended to a basis $\mathcal B$ of $\mathcal D^*(\Omega)$ by Theorem~\ref{T: Existence of Basis}. We define $U\in\mathcal D^*(\Omega)$ by $U(\psi)=\Phi(\psi)$ for all $\psi\in\mathcal B_1$ and $U(\psi)=0$ (or anyhow) for $\psi\in\mathcal B\setminus\mathcal B_1$. It is clear that $U$ is an extension of $\Phi$ from $\ran(P(x, \partial))$ to $\mathcal D^*(\Omega)$. Thus $\langle P^*(x, \partial)U, \varphi\rangle=\langle U,\,  {P(x, \partial)}\varphi\rangle=\langle \Phi,\,  {P(x, \partial)}\varphi\rangle=\langle T, \varphi\rangle,$ for all $\varphi\in\mathcal D(\Omega)$, as required. \qquad\qquad\qquad\qquad\qquad\qquad$\square$
\end{description} 
\section{Schwartz Distributions within $\mathcal D^*(\Omega)$:\\ Sequential Approach to Distribution Theory}\label{S: Schwartz Distributions within D*(Omega)}

We characterize the space of Schwartz distributions  $\mathcal D^\prime(\Omega)$ (Vladimirov~\cite{vVladimirov}) as a particular subspace of  $\mathcal D^*(\Omega)$ without involving the usual strong topology on the space of test-functions $\mathcal D(\Omega)$ (Vladimirov~\cite{vVladimirov}). We discuss the similarities and differences between $\mathcal D^\prime(\Omega)$ and $\mathcal D^*(\Omega)$ and give a short outline of a sequential approach to distribution theory based on our characterization. 

\begin{theorem}[Connection to Schwartz Distributions]\label{T: Connection to Schwartz Distributions}
	\begin{T-enum}

		\item\label{Item: Characterization} 
		The space $\mathcal D^\prime(\Omega)$ of \linebreak \textbf{Schwartz distributions} on $\Omega$ (Vladimirov~\cite{vVladimirov}) coincides with the \textbf{weak sequential completion} of $\mathcal L_{loc}(\Omega)$ within $\mathcal D^*(\Omega)$, i.e.
		\begin{equation}\label{E: Distributions}
			\mathcal D^\prime(\Omega)=\big\{T\in\mathcal D^*(\Omega): (\exists (T_n)\in \big(S_\Omega[\mathcal L_{loc}(\Omega)]\big)^\mathbb N)(T_n\overset{*}\mapsto T)\big\},
		\end{equation}
		where $S_\Omega[\mathcal L_{loc}(\Omega)]$  is the image of $\mathcal L_{loc}(\Omega)$ under $S_\Omega$ (Definition~\ref{D: The Space D*(Omega)}) and $\big(S_\Omega[\mathcal L_{loc}(\Omega)]\big)^\mathbb N$  denote the space of all sequences in $S_\Omega[\mathcal L_{loc}(\Omega)]$ (compare with Remark~\ref{R: Sequential Approach to Distribution Theory} below).
		
		\item  $\mathcal D^\prime(\Omega)$ is a \textbf{differential $\mathcal E(\Omega)$-submodule} of $\mathcal D^*(\Omega)$. Consequently, $\mathcal D^\prime(\Omega)$ is a \textbf{differential $\C$-vector subspace} of $\mathcal D^*(\Omega)$.
		
		\item The family $\{\mathcal D^\prime(\Omega)\}_{\Omega\in\mathcal T^d}$ is a \textbf{subsheaf} of 
		$\{\mathcal D^*(\Omega)\}_{\Omega\in\mathcal T^d}$ of differential $\mathcal E(\Omega)$-modules (and $\C$-vector spaces) (Definition~\ref{D: The Space D*(Omega)}), where $\mathcal T^d$ stands for the usual topology on $\R^d$. Consequently, the \textbf{support} $\supp(T)$ in $\mathcal D^*(\Omega)$ coincides with the usual support of $T$ in distribution theory (Vladimirov~\cite{vVladimirov}, \S 1.5, p.16) for every distribution $T$.
		
		\item The inclusion $\mathcal D^\prime(\Omega)\subset\mathcal D^*(\Omega)$ is \textbf{invariant under diffeomorphisms} $\theta\in\Diff(X, Y)$.
	\end{T-enum}
\end{theorem}

\begin{proof}
	\begin{Pf-enum}
		\item follows from the fact that the space of Schwartz distributions is \emph{sequentially complete} under the weak convergence and every distribution can be \emph{regularized} within $\mathcal D(\Omega)$. For a detailed proof we refer the reader to   (Vladimirov~\cite{vVladimirov}, \S 1.4, p.14 and \S 4.6, p. 80--81). 
		
		\item The product $f\,T$ in $\mathcal E(\Omega)\times\mathcal D^*(\Omega)$ (Definition~\ref{D: The Space D*(Omega)}) coincides with the product $f\,T$ in $\mathcal E(\Omega)\times\mathcal D^\prime(\Omega)$ in the case when $T$ is a distribution. 
		
		\item The definition of \emph{restriction} $T\rest X$ in $\mathcal D^*(\Omega)$ (Definition~\ref{D: The Space D*(Omega)}) coincides with the definition of \emph{restriction} $T\rest X$ in Schwartz theory of distributions (Vladimirov~\linebreak\cite{vVladimirov}, \S 1.3, p.12) in the case when $T$ is a distribution. 
		
		\item The definition of \emph{change of variables} $Q_*(T)$ in $\mathcal D^*(\Omega)$ (Definition~\ref{D: The Space D*(Omega)}) coincides with the definition of \emph{change of variables}  $Q_*(T)$ in Schwartz theory of distributions (H\"{o}rmander\cite{lHor-I}, \S 6.1) in the case when  $T$ is a distribution. 
		\end{Pf-enum}
\end{proof}
\begin{remark}[Sequential Approach to Distribution Theory]\label{R: Sequential Approach to Distribution Theory} 
	\begin{R-enum}
		\item The formula (\ref{E: Distributions}) can be written in the form:
		\begin{align}\label{E: Distributions Again}
			\mathcal D^\prime(\Omega)=\Big\{T\in\mathcal D^*(\Omega): &(\exists (f_n)\in\big(\mathcal L_{loc}(\Omega)\big)^\mathbb N)(\forall\varphi\in\mathcal D(\Omega))\\
			&\lim_{n\mapsto\infty}\int_\Omega f_n(x)\varphi(x)\, dx=\bra T, \varphi\ket\Big\}\notag,
		\end{align}
		where $\big(\mathcal L_{loc}(\Omega)\big)^\mathbb N$ denotes the space of all sequences in $\mathcal L_{loc}(\Omega)$ (Section~\ref{S: Notations}) and the ``$\lim$'' stands for the usual limit in $\C$. 
		
		\item  The space $\mathcal L_{loc}(\Omega)$ in (\ref{E: Distributions Again}) can be replaced by $\mathcal E(\Omega)=\mathcal C^\infty(\Omega)$ and even by $\mathcal D(\Omega)=\mathcal C_0^\infty(\Omega)$ (Vladimirov~\cite{vVladimirov}, \S 4.6, p.79). This gives rise to the following \emph{sequential definition of Schwartz distributions} $\mathcal D^\prime(\Omega)=F(\mathcal E(\Omega)^\mathbb N)/_\sim$. Here $F(\mathcal E(\Omega)^\mathbb N)$ denotes the set of all \emph{fundamental sequences} $(f_n)$ in $\mathcal E(\Omega)$, i.e. the sequences with the property that for every (fixed) test function $\varphi\in\mathcal D(\Omega)$ the sequence $\big(\int_\Omega f_n(x)\varphi(x)\, dx\big)$ is  fundamental (convergent) in $\C$. Also,  $\sim$ stands for the equivalence relation on $F(\mathcal E(\Omega)^\mathbb N)$ defined by $(f_n)\sim (g_n)$ if $\lim_{n\mapsto\infty}\int_\Omega \big(f_n(x)-g_n(x)\big)\varphi(x)\, dx=0$ for all $\varphi\in\mathcal D(\Omega)$. For a similar sequential approach, we refer to Lighthill~\cite{mjLighthill}.
	\end{R-enum}
\end{remark}

\begin{theorem}\label{T: D*-D' non-empty} $\mathcal D^*(\Omega)\setminus\mathcal D^\prime(\Omega)\not=\varnothing$.
\end{theorem}

\begin{description}
	\item[$\mathnormal{Proof\, 1}$]  $\mathcal D^\prime(\Omega)$ is a vector subspace of $\mathcal D^*(\Omega)$, and $\dim(\mathcal D^\prime(\Omega))=
	\mathfrak{c}$, and\linebreak $\dim(\mathcal D^*(\Omega))
	=\mathfrak{c}_+$  (Example~\ref{Ex: D(Omega) and its Dual}). Thus $\mathcal D^\prime(\Omega)$ is a proper subspace of $\mathcal D^*(\Omega)$ by Lemma~\ref{L: Subspace Lemma}.
	
	\item[$\mathnormal{Proof\, 2}$]    Let $\{\varphi_0, \varphi_1, \varphi_2,\dots\}$ be a free set of $\mathcal D(\Omega)$ such that $\varphi_n\mapsto \varphi_0$ as $n\mapsto\infty$, in the strong topology of $\mathcal D(\Omega)$ (Vladimirov~\cite{vVladimirov}, \S 1.2, p.7). Since $\dim(\mathcal D(\Omega))=\mathfrak{c}$,  there exists a (Hamel) basis $\{\varphi_r: r\in\R\}$ of $\mathcal D(\Omega)$ which extends the sequence $\{\varphi_0, \varphi_1, \varphi_2,\dots\}$ by Theorem~\ref{T: Existence of Basis}. Define $T\in\mathcal D^*(\Omega)$ by $T(\varphi_0)=1$ and $T(\varphi_r)=0$,\, $r\in\R,\, r\not=0$. Thus $T\in\mathcal D^*(\Omega)\setminus\mathcal D^\prime(\Omega)$ (as required) by (Theorem~\ref{T: Connection to Schwartz Distributions}, \ref{Item: Characterization}), since 
	$\lim_{n\mapsto\infty}T(\varphi_n)=0\not=1= T(\varphi_0)$.
	
	\item[$\mathnormal{Proof\, 3}$]    For the existence of  discontinuous linear functionals (based on a Borel theorem), we refer to (Oberguggenberger~\cite{MO13}, Example~10, p. 11).
\end{description}
\qquad\qquad\qquad\qquad\qquad\qquad\qquad\qquad\qquad\qquad\qquad\qquad\qquad\qquad\qquad$\square$\newline
\indent We borrow the next definition and the following lemma and theorem from (Oberguggenberger~\cite{MO13} p.15).

\begin{definition}[Convolution]\label{D: Convolution} Let $S\in\mathcal D^*(\R^d)$ or $S\in\mathcal E^*(\R^d)$. Let $T\in\mathcal D^\prime(\R^d)$. We define the \textbf{convolution} $S\star T\in\mathcal D^*(\R^d)$ by 
	\[
	\langle S\star T, \varphi\rangle= \langle S, \Check{T}\star\varphi\rangle,
	\]
	for all $\varphi\in\mathcal D(\R^d)$, where $\Check{T}$ is the \emph{inflection} of $T$ (Example~\ref{Ex: Inflections and Translations}) and $\Check{T}*\varphi$ is the \emph{usual convolution} in the sense of distribution theory (Vladimirov~\cite{vVladimirov}, Ch. 4).  
\end{definition}

\begin{lemma}  $\partial^\alpha(S\star T)=(\partial^\alpha S)\star T= S\star (\partial^\alpha T)$ for all multi-indexes $\alpha\in\mathbb N^d_0$. 
\end{lemma}

\begin{proof}
	Let $\varphi\in\mathcal D(\R^d)$. We calculate
	$\langle\partial^\alpha(S\star T), \varphi\rangle=(-1)^{|\alpha|}\langle S, \Check{T}\star\partial^\alpha\varphi\rangle=(-1)^{|\alpha|}\langle S, \partial^\alpha(\Check{T}\star\varphi)\rangle$. On the other hand, $(-1)^{|\alpha|}\langle S, \partial^\alpha(\Check{T}\star\varphi)\rangle=\langle(\partial^\alpha S)\star T, \varphi\rangle$. Thus $\partial^\alpha(S\star T)=(\partial^\alpha S)\star T$ as required. Also, $(-1)^{|\alpha|}\langle S, \partial^\alpha(\Check{T}\star\varphi)\rangle=(-1)^{|\alpha|}\langle S,(\partial^\alpha \Check{T})\star\varphi\rangle=\langle S, \Check{(\partial^\alpha T)}\star\varphi)\rangle$=$\langle S\star\partial^\alpha T, \varphi)\rangle$. Thus $\partial^\alpha(S\star T)=S\star\partial^\alpha T$ as required. 
\end{proof}
\begin{theorem}[Fundamental Solutions]\label{T: Fundamental Solutions}
	Let $c_\alpha\in\C$, $\alpha\in\mathbb N_0^d$, $|\alpha|\leq m$ and $P^*(\partial): \mathcal D^*(\R^d)\mapsto\mathcal D^*(\R^d)$, $P^*(\partial)=\sum_{|\alpha|\leq m}c_\alpha\, \partial^\alpha$, be the corresponding linear partial differential operator with \textbf{constant coefficients}. Then: 
	
	\begin{T-enum}
		\item The operator $P^*(\partial)$ has a \textbf{fundamental solution} $F\in\mathcal D^*(\Omega)$, a solution of the equation $P^*(\partial)\,F=\delta$ in $\mathcal D^*(\Omega)$.
		
		\item Let $F\in\mathcal D^*(\Omega)$ be a fundamental solution of $P^*(\partial)$ and $T\in\mathcal E^\prime(\Omega)$. Then $U=F\star T$ is a solution of $P^*(\partial)\,U=T$ in  $\mathcal D^*(\Omega)$.
		
		\item Let $F\in\mathcal D^\prime(\Omega)$ be a fundamental solution of $P^*(\partial)$ and $T\in\mathcal D^*(\Omega)$ (or even $T\in\mathcal E^*(\Omega)$). Then $U=F\star T$ is a solution of $P^*(\partial)\,U=T$ in  $\mathcal D^*(\Omega)$.
	\end{T-enum}
\end{theorem}
\begin{proof}
	\begin{Pf-enum}
		\item follows directly from Theorem~\ref{T: The Main Result}, because $P^*(\partial)$ is a regular operator (Section~\ref{S: Examples of Regular Operators on D*(Omega): Solvable Partial Differential Equations}) and $\mathcal D^\prime(\Omega)\subset\mathcal D^*(\Omega)$ (thus $\delta\in\mathcal D^*(\Omega)$). In both (i) and (ii) we have $P^*(\partial)U=P^*(\partial)(F\star T)=(P^*(\partial)F)\star T)=\delta\star T=T$. The last equality, $\delta\star T=T$, is derived in (ii) and (iii) somewhat differently:
		
		\item $\langle T\star\delta,\varphi\rangle=\langle \delta, \Check{T}\star\varphi\rangle=(\Check{T}\star\varphi)(0)=\langle T,\varphi\rangle$, for all $\varphi\in\mathcal D(\R^d)$.

		\item $\langle T\star\delta,\varphi\rangle=\langle T, \Check{\delta}\star\varphi\rangle=\langle T, \delta\star\varphi\rangle=\langle T,\varphi\rangle$, for all $\varphi\in\mathcal D(\R^d)$ (or even for all $\varphi\in\mathcal E(\R^d)$).
	\end{Pf-enum}
\end{proof}
\begin{remark}[Comparison: $\mathcal D^\prime(\Omega)$ vs. $\mathcal D^*(\Omega)$]\label{R: Comparison: Dprime(Omega) vs. D*(Omega)}
	
	There are obvious \emph{similarities} between the properties of  $\mathcal D^*(\Omega)$ (Theorem~\ref{T: Properties of D*(Omega)}) and the space of Schwartz distributions $\mathcal D^\prime(\Omega)$ Vladimirov~\cite{vVladimirov} (see also Remark~\ref{R: Sequential Approach to Distribution Theory} in this paper). However, there are also \emph{essential differences}; here are some of them:
	\begin{enumerate}
		\item The discontinuous (relative to the strong topology on $\mathcal D(\Omega)$) linear functionals $T\in\mathcal D^*(\Omega)\setminus\mathcal D^\prime(\Omega)$ cannot be approached by a sequence of classical functions in the sense that there is \emph{no sequence} $(f_n)$ in $\mathcal L_{loc}(\Omega)$ such that $f_n\overset{*}{\mapsto}T$ (Definition~\ref{D: The Space D*(Omega)}). In a sense the spaces $\mathcal L_{loc}(\Omega)$ (more precisely, $S[\mathcal L_{loc}(\Omega)]$), $\mathcal D^\prime(\Omega)$ and $\mathcal D^*(\Omega)$ resemble $\mathbb Q$, $\R$ and $\C$, respectively: Every real number is the limit of some (fundamental) sequence in $\mathbb Q$, but the complex numbers of the for form $a+ib,\, b\not=0$, can not be approximated by sequences in $\mathbb Q$ (we are unaware of subspace of $\mathcal D^*(\Omega)$ which plays the role of $\mathbb Q(i)$ in the above anlalogy).
		
		\item The \emph{structural theorem} for $\mathcal D^\prime(\Omega)$ (Vladimirov~\cite{vVladimirov}, \S 2.4, p.41) fails in $\mathcal D^*(\Omega)$. Recall that the structural theorem states that for every Schwartz distribution $T\in\mathcal D^\prime(\Omega)$ and for every open set $X$ of $\R^d$, such that $X\subset\subset\Omega$, there exist a classical function $f\in\mathcal L^\infty(X)$ and a multi-index $\alpha\in\mathbb N_0^d$ such that $T=\partial^\alpha f$ in $\mathcal D^\prime(X)$. This theorem fails in $\mathcal D^*(\Omega)$.
		
		\item The convolution (Definition~\ref{D: Convolution}) \emph{fails to regularize} $T\in\mathcal D^*(\R^d)\setminus\mathcal D^\prime(\R^d)$
		in the sense that $T\star\varphi\in\mathcal E(\R^d)$ \emph{does not necessarily hold} for all $\varphi\in\mathcal D(\R^d)$.
		
		\item The direct (tensor) product in $\mathcal D^\prime(\Omega)$(Vladimirov~\cite{vVladimirov}, \S 3, p.46) also \emph{fails} in $\mathcal D^*(\Omega)$.
		
		\item From the above list it seems that the space $\mathcal D^\prime(\Omega)$ is superior over $\mathcal D^*(\Omega)$ at least from the point of view of partial differential operator theory. As we shall see in the next sections however, the regular operators $P^*(x, \partial)$ are (always) \emph{surjective} on $\mathcal D^*(\Omega)$, but their restrictions on $\mathcal D^\prime(\Omega)$ are often \emph{not}. That means that the partial differential equations of the form $P^*(x, \partial)U=T$ are always \emph{solvable} for $U$ in $\mathcal D^*(\Omega)$, but \emph{not necessarily solvable} in $\mathcal D^\prime(\Omega)$. The latter property of $\mathcal D^*(\Omega)$ is the main reason why we believe that the space $\mathcal D^*(\Omega)$ - rather than $\mathcal D^\prime(\Omega)$ - should be considered as  the \emph{natural framework} of partial differential equations, especially the linear ones with smooth coefficients. 
	\end{enumerate}
\end{remark}
\section{Three Invariant Subspaces}\label{S: Three Invariant Subspaces}

In order to prepare better for the discussion in the next section, we select three important subspaces of $\mathcal D^*(\Omega)$, which are invariant under the linear partial differential operators with $\mathcal C^\infty$-coefficients.

\begin{lemma}[Three Invariant Spaces]\label{L: Three Invariant Spaces} Let $P^*(x, \partial)\!: \mathcal D^*(\Omega)\mapsto\mathcal D^*(\Omega)$, $
	P^*(x, \partial)=\sum_{|\alpha|\leq m}c_\alpha(x)\partial^\alpha$, be a linear partial differential operator (regular or not) with $\mathcal C^\infty$-coefficients $c_\alpha\in\mathcal E(\Omega)$. The subspaces of the test-functions $\mathcal D(\Omega)=\mathcal C_0^\infty(\Omega)$, the $\mathcal C^\infty$-functions $\mathcal E(\Omega)=\mathcal C^\infty(\Omega)$ and the Schwartz distributions $\mathcal D^\prime(\Omega)$ are all invariant under 
	$P^*(x, \partial)$.
\end{lemma}
\begin{proof} An immediate consequence from fact that the multiplication by a smooth function on $\mathcal D^*(\Omega)$ (Definition~\ref{D: The Space D*(Omega)}, \#\ref{Item: Product}) and  differentiation on $\mathcal D^*(\Omega)$ (Definition~\ref{D: The Space D*(Omega)}, \#\ref{Item: Differentiation}) both coincide on $\mathcal E(\Omega)$ with the usual multiplication and usual differentiation; and on  $\mathcal D^\prime(\Omega)$ - with the operations with the same names in the sense of Schwartz theory of distributions (Vladimirov~\cite{vVladimirov}).
\end{proof}
\section{Examples of Regular Operators on $\mathcal D^*(\Omega)$: Solvable Partial Differential Equations}\label{S: Examples of Regular Operators on D*(Omega): Solvable Partial Differential Equations}

Following Todorov~\cite{tdTod95}, Oberguggenberger and Todorov~\cite{MO&TDT} and Oberguggenberger~\linebreak\cite{MO13}, we present several examples of \emph{regular operators} (Definition~\ref{D: Regular Operators}) of the form $P^*(x, \partial): \mathcal D^*(\Omega)\mapsto\mathcal D^*(\Omega)$, $
P^*(x, \partial)=\sum_{|\alpha|\leq m}c_\alpha(x)\partial^\alpha$, with $\mathcal C^\infty$-coefficients $c_\alpha\in\mathcal E(\Omega)$. 
Recall that $P^*(x, \partial)$ is called regular if its co-dual $P(x, \partial)\!: \mathcal D(\Omega)\mapsto\mathcal D(\Omega)$,\, $P(x, \partial)\varphi(x)=\sum_{|\alpha|\leq m}(-1)^{|\alpha|} \partial^\alpha\big(c_\alpha(x)\varphi(x)\big)$, is injective on $\mathcal D(\Omega)$ (Theorem~\ref{T: Properties of D*(Omega)}, \ref{Item: Regular}). Recall as well that the regular operators are exactly the operators for which our main result (Theorem~\ref{T: The Main Result}) holds, i.e. these are \emph{surjective operators on} $\mathcal D^*(\Omega)$.

The restrictions $P^*(x, \partial)\rest\mathcal D(\Omega)$, $P^*(x, \partial)\rest\mathcal E(\Omega)$ and  $P^*(x, \partial)\rest\mathcal D^\prime(\Omega)$ of the regular operators $P^*(x, \partial)$ in our examples below are relatively well-studied in the classical theory of linear partial differential operators (H\"{o}rmander~\cite{lHor-I}-\cite{lHor-IV}). With few exceptions, these restrictions  \emph{are non-surjective}. Thus we arrived at the \emph{main point of our approach} - to extend non-surjective operators from $\mathcal D^\prime(\Omega)$ to surjective operators on $\mathcal D^*(\Omega)$ and thus to guarantee the \emph{existence of solutions for the corresponding partial differential equations} in the framework of $\mathcal D^*(\Omega)$. 

\begin{example}[Constant Coefficients]\label{Ex: Constant Coefficients} The \emph{linear partial differential operators with constant coefficients}: $P^*(\partial): \mathcal D^*(\Omega)\mapsto\mathcal D^*(\Omega)$,  where $P^*(\partial)=\sum_{|\alpha|\leq m}c_\alpha\, \partial^\alpha$,  are \emph{regular}. To show that, we have to show that its co-dual (transposed)  operator $P(\partial): \mathcal D(\Omega)\mapsto\mathcal D(\Omega)$,  $P(\partial)=
	\sum_{|\alpha|\leq m}c_\alpha\, (-1)^{|\alpha|}\, \partial^\alpha$, is injective. Indeed, the Fourier transform applied to $P(\partial)\varphi=0$ leads us to $\big(\sum_{|\alpha|\leq m}c_\alpha\, (iz)^\alpha\big)\mathcal F(\varphi)(z)=0$.  Now, both $\sum_{|\alpha|\leq m}c_\alpha\, (iz)^\alpha$ and $\mathcal F(\varphi)$ are entire functions and thus the function $\sum_{|\alpha|\leq m}c_\alpha\, (iz)^\alpha$ can be cancelled, because the ring of entire functions forms an integral domain (a ring without zero-divisors). The result is $\mathcal F(\varphi)(z)=0$, which implies $\varphi=0$ (as required). Consequently, $P^*(\partial)$ is \emph{surjective} on $\mathcal D^*(\Omega)$, i.e. the equation $P^*(\partial)U=T$ has a solution for $U$ in $\mathcal D^*(\Omega)$ for every choice of $T$ also in $\mathcal D^*(\Omega)$. Recall that the restrictions $P^*(\partial)\rest\mathcal D(\Omega)$, $P^*(\partial)\rest\mathcal E(\Omega)$ and $P^*(\partial)\rest\mathcal D^\prime(\Omega)$ are \emph{not necessarily surjective}. In more detail the situation is as follows:
	
	\begin{itemize}
		\item In the particular case $\Omega=\R^d$  the operators $P^*(\partial)\rest\mathcal D(\R^d)$ \emph{is  neither surjective, nor injective} (think of harmonic functions), but the operator  $P^*(\partial)\rest\mathcal E(\R^d)$ \emph{is surjective}, since every convex open subset of $\R^d$ is $P^*$-convex for supports (H\"{o}rmander~\cite{lHor-II}, Theorem 10.6.2). The operator $P^*(\partial)\rest\mathcal D^\prime(\R^d)$ 
		\emph{is also surjective}: this is the famous existence theorem of Malgrange~\cite{bMal55-56} and Ehrenpreis~\cite{lEhren56}. 
		
		\item If $\Omega\not=\R^d$ however, the operators $P^*(\partial)\rest\mathcal D^\prime(\Omega)$ might be also \emph{non-surjective}. In particular, if $P^*(\partial)\!\!\upharpoonright\!\mathcal D^\prime(\Omega)$ is \emph{non-elliptic}, then there exists  an open set $\Omega$ of $\R^d$ (which is not \emph{$P^*$-convex for supports}) and a smooth function $f\in\mathcal E(\Omega)$ such that the equation $P^*(\partial)U=f$ has \emph{no a distributional solution} for $U$ (H\"{o}rmander~\cite{lHor-II}, Theorem 10.6.6, Corollary 10.6.8, Theorem 10.8.2). By our result (Theorem~\ref{T: The Main Result}) however, $P^*(\partial)U=f$ still does have a solution $U\in\mathcal D^*(\Omega)\setminus\mathcal D^\prime(\Omega)$. In addition, we observe that $P^*(\partial)$ is certainly non-hypoelliptic on $\mathcal D^*(\Omega)$ even if its restriction $P^*(\partial)\rest\mathcal D^\prime(\Omega)$ is hypoelliptic, because the solution  $U$ is non-smooth.
		
		\item For a recent discussion on hypoellipticity we refer to Street~\cite{bStreet18}.
	\end{itemize}
\end{example}
\begin{example}[Elliptic Operators]\label{Ex: Elliptic Operators} All \emph{second order elliptic operators with $\mathcal C^\infty$-coefficients} $P^*(x, \partial): \mathcal D^*(\Omega)\mapsto\mathcal D^*(\Omega)$ are \emph{regular}. This follows from the fact that every test-function $\varphi\in\mathcal D(\Omega)$ vanishes on an open set $\Omega$ of $\R^d$ and thus $P(x, \partial)\varphi=0$ implies $\varphi=0$ by the Cauchy uniqueness property (uniqueness continuation principle), since $P(x, \partial)$ is also a second order elliptic operator with $\mathcal C^\infty$-coefficients. Consequently, $P^*(x, \partial)$ is \emph{surjective} on $\mathcal D^*(\Omega)$, i.e. the equation $P^*(x, \partial)U=T$ has a solution for $U$ in $\mathcal D^*(\Omega)$ for every choice of $T$ also in $\mathcal D^*(\Omega)$. We should mention that the restrictions $P^*(x, \partial)\rest\mathcal 
	D(\Omega)$, $P^*(x, \partial)\rest\mathcal E(\Omega)$ and $P^*(x, \partial)\rest\mathcal D^\prime(\Omega)$ \emph{are not necessarily surjective}. For more details we refer to (H\"{o}rmander~\cite{lHor-I}, 8.5-8.6), (H\"{o}rmander~\cite{lHor-III}, 17.2) and (H\"{o}rmander~\cite{lHor-IV}, 28.1-28.4).
\end{example}
\begin{example}[Analytic Coefficients]\label{Ex: Analytic Coefficients} All \emph{elliptic operators with analytic coefficients} $P^*(x, \partial)\!\!:\mathcal D^*(\Omega)\mapsto\mathcal D^*(\Omega)$ are \emph{regular} since their co-dual (transposed) operators $P(x, \partial)\!\!: \mathcal D(\Omega)\mapsto\mathcal D(\Omega)$ are also \emph{elliptic with analytic coefficients}. More generally, all linear partial differential operators $P^*(x, \partial)\!\!: \mathcal D^*(\Omega)\mapsto\mathcal D^*(\Omega)$  \emph{with analytic coefficients of constant strength} are regular (see the discussion before Theorem 13.5.2 in H\"{o}rmander~\cite{lHor-II}, p.196).  Consequently, $P^*(x, \partial)$ is \emph{surjective} on $\mathcal D^*(\Omega)$, i.e. the equation $P^*(x, \partial)U=T$ has a solution for $U$ in $\mathcal D^*(\Omega)$ for every choice of $T$ also in $\mathcal D^*(\Omega)$. We should note that the restrictions  $P^*(x, \partial)\rest\mathcal D(\Omega)$, $P^*(x, \partial)\rest\mathcal E(\Omega)$ and $P^*(x, \partial)\rest\mathcal D^\prime(\Omega)$ \emph{are not necessarily surjective}.  For the \emph{local solvability} of equations of the form $P^*(x, \partial)U=T$ (which is quite a different phenomenon) and for a more comprehensive discussion on the solvability in general, we refer to Treves~\cite{fTreves70}. 
\end{example}
\begin{example}[Hans Lewy Operator]\label{Ex: Hans Lewy Operator}
	\begin{itemize}
		\item The \emph{Hans Lewy~\cite{hLewy57} operator}:\linebreak 
		$L^*(x, \partial)\!:\mathcal D^*(\R^3)\mapsto\mathcal D^*(\R^3)$, $
		L^*(x, \partial)=\frac{\partial}{\partial x_1}+i \frac{\partial}{\partial x_2}-2i(x_1+ix_2)\frac{\partial}{\partial x_3}$,
		is \emph{regular} by the following reasoning (also due to H. Lewy): First, we verify that $L(x, \partial)=-L^*(x, \partial)$. Next, we observe that the mapping $(x_1, x_2, x_3)\mapsto \big(x_1+ix_2, x_3+i(x_1^2+x_2^2)\big)$ is an embedding of $\R^3$ into $\C^2$, which converts $\R^3$ into the boundary of the domain $D=\{{\rm Im} (z_2) >|z_1|^2\}$. Thus every solution $\varphi\in\mathcal D(\R^3)$ of $L(x, \partial)\varphi=0$ can be extended holomorphically into $D$. Thus $\varphi=0$ by the unique continuation for holomorphic functions. Consequently, $L^*(x, \partial)$ is \emph{surjective} on $\mathcal D^*(\R^3)$, i.e. the equation $L^*(x, \partial)U=T$ has a solution for $U$ in $\mathcal D^*(\R^3)$ for every choice of $T$ also in $\mathcal D^*(\R^3)$. \emph{By contrast}, recall  that the Lewy operator $L^*(x, \partial)\rest\mathcal D(\R^3)$, $L^*(x, \partial)\rest\mathcal E(\R^3)$ and $L^*(x, \partial)\rest\mathcal D^\prime(\R^3)$ are \emph{non-surjective} and the equation $L^*(x, \partial)U=T$ might fail to have solutions $U$ in $\mathcal D^\prime(\R^3)$ even for some test-functions $T\in\mathcal D(\R^3)$ (Lewy~\cite{hLewy57}).
		
		\item With different arguments (due to Oberguggenberger~\cite{MO13}, Remark 20(c), p.18) we show that  the Lewy operator $L^*(x, \partial)\!\!:\mathcal D^*(\Omega)\mapsto\mathcal D^*(\Omega)$ is \emph{regular for arbitrary open set} $\Omega$ of $\R^3$ (not only for $\Omega=\R^3$): We apply the Fourier transform $\mathcal F_3$ - relative to the variable $x_3$ - to $L(x, \partial)\varphi(x_1, x_2, x_3)=0$. The result is $(-\frac{\partial}{\partial x_1}-i \frac{\partial}{\partial x_2}+2(x_1+i x_2)z_3)\psi(x_1, x_2, z_3)=0$, where $\psi(x_1, x_2, z_3)=\mathcal F_3(\varphi)$ is an entire function in $z_3$ (for any fixed $(x_1, x_2)\in\R^2$ such that $(x_1, x_2, x_3)\in\Omega$ for some $x_3\in\R$). Next, we let $x_1+ix_2=z_2$ and the result is $\big(\frac{\partial}{\partial\overline{z_2}}+2z_2z_3\big)\psi(z_2, z_3)=0$. The latter implies that  the function $z_2\mapsto\psi(z_2, 0)$ is analytic and thus (having a compact support) vanishes identically in the variable $z_2$. On the other hand, by differentiating $\big(\frac{\partial}{\partial\overline{z_2}}+2z_2z_3\big)\psi(z_2, z_3)=0$ consecutively and letting $z_3=0$, we get $\frac{\partial^k}{\partial_{z_3}}\psi(z_2, 0)=0$ for all $z_2$ and all $k\in\mathbb N_0$. Thus, $\psi=0$ (since, again, $z_3\mapsto\psi(z_2, z_3)$ is  analytic). The latter implies $\varphi=0$ proving that $L^*(x, \partial)$ is regular. Consequently, $L^*(x, \partial)$ is \emph{surjective} on $\mathcal D^*(\Omega)$, i.e. the equation $L^*(x, \partial)U=T$ has a solution for $U$ in $\mathcal D^*(\Omega)$ for every choice of $T$ also in $\mathcal D^*(\Omega)$ and any open $\Omega$. By contrast,  $L^*(x, \partial)\rest\mathcal D^\prime(\Omega)$ is not necessarily surjective (see above). 
	\end{itemize}
\end{example}

\begin{examples}[Non-Regular Operators]\label{Ex: Non-Regular Operators} Here are several examples of linear partial differential operators with smooth coefficients which are \emph{non-regular}:
	\begin{Ex-enum}
		
		\item Let $\psi\in\mathcal D(\Omega)$. Then the operator $P^*(x, \partial)\!\!: \mathcal D^*(\Omega)\mapsto \mathcal D^*(\Omega)$, $P^*(x, \partial)=\psi\partial^\alpha$ is \emph{non-regular}, because $P(x, \partial)\varphi= (-1)^{|\alpha|}\partial^\alpha(\psi\varphi)$ and thus $P(x, \partial)\varphi=0$ for every 
		$\varphi\in\mathcal D(\Omega)$,  $\varphi\not= 0$, with $\supp(\varphi)\cap\supp(\psi)=\varnothing$.
		
		\item The operator $P^*(x, \partial)\!: \mathcal D^*(\R^2)\mapsto \mathcal D^*(\R^2)$, $P^*(x, \partial)=x_1\frac{\partial}{ \partial x_2} -x_2\frac{\partial}{\partial x_1}$ is \emph{non-regular}, because $P(x, \partial)= - P^*(x, \partial)$ and $P(x, \partial)\varphi(x_1^2+x_2^2)=0$ for every $\varphi\in\mathcal D(\R)$,  $\varphi\not= 0$.
		
		\item For an example of a \emph{non-injective fourth order  elliptic operator $P(x,\partial)\!\!: \mathcal D(\R^3)\mapsto\mathcal D(\R^3)$ with smooth (but non-analytic) coefficients}, we refer to (H\"{o}rmander~\cite{lHor-II}, Theorem 13.6.15). Its dual $P^*(x, \partial)\!\!:  \mathcal D^*(\R^3)\mapsto\mathcal D^*(\R^3)$ gives another \emph{example for a non-regular operator}. 
	\end{Ex-enum}
\end{examples}
\section{Standardization of a Non-Standard Result}\label{S: Standardization of a Non-Standard Result}

\indent Using Hamel bases, we show that the space of generalized distributions $\widehat{\mathcal E}(\Omega)$ introduced in (Todorov~\cite{tdTod95}, \S 2) can be embedded as a $\C$-vector subspace into $\mathcal D^*(\Omega)$ (Example~\ref{Ex: D(Omega) and its Dual}). Because $\widehat{\mathcal E}(\Omega)$ was defined in the framework of non-standard analysis (Robinson~\cite{aRob66}), we look upon $\mathcal D^*(\Omega)$ as a  \emph{standardization} of $\widehat{\mathcal E}(\Omega)$. Actually, our article itself can be viewed as a \emph{standardization} of the article Todorov~\cite{tdTod95}, because the surjectivity of the regular operators was first proved in Todorov~\cite{tdTod95} in the framework of $\mathcal L(\widehat{\mathcal E}(\Omega))$, while the main result of this article (Theorem~\ref{T: The Main Result}) holds within $\mathcal L(\mathcal D^*(\Omega))$. Thus, by replacing $\widehat{\mathcal E}(\Omega)$ by $\mathcal D^*(\Omega)$, our result about the regular operators becomes accessible even for readers without background in non-standard analysis. 

\begin{example}[The Space $\widehat{\mathcal E}(\Omega)$]\label{Ex: The Space hatE(Omega)} Let $^*\mathcal E(\Omega)$ be the non-standard extension of $\mathcal E(\Omega)=\mathcal C^\infty(\Omega)$ in a $\mathfrak{c}_+$-saturated ultra-power non-standard model with set of individuals $\R$ (Example~\ref{Ex:  The Space *E(Omega)|C and its Dual} applied for $\kappa=\mathfrak{c}=\card\, \R$).
	Let $\widehat{\mathcal E}(\Omega)={^*\mathcal E}(\Omega)/J_{\mathcal D(\Omega)}$ be the space 
	of \emph{generalized distributions}, where
	\[
	J_{\mathcal D(\Omega)}=\Big\{f\in {^*\mathcal E}(\Omega): (\forall\varphi\in\mathcal D(\Omega))\big(\int_\Omega f(x)\, {^*\!\varphi}(x)\, dx=0\big)\Big\},
	\]
	and $^*\!\varphi$ is the non-standard extension of $\varphi$.
	\begin{itemize}
		\item Recall that $^*\mathcal E(\Omega)$ is a differential ring and  $^\sigma\!\mathcal E(\Omega)=\{^*\!\varphi: \varphi\in\mathcal E(\Omega)\}$ is a differential subring of $^*\mathcal E(\Omega)$. For our purpose we shall treat $^*\mathcal E(\Omega)$ as a differential module over $\mathcal E(\Omega)$ under the product $\varphi f={^*\!\varphi}\cdot f$, where $\varphi\in\mathcal E(\Omega)$ and  $f\in{^*\mathcal E}(\Omega)$. Here the dot in ${^*\!\varphi}\cdot f$ stands for the point-wise product in $^*\mathcal E(\Omega)$. Consequently, $^*\mathcal E(\Omega)$ is a differential vector space over $\C$ as well, since $\C\subset\mathcal E(\Omega)$. We have (trivially) $\card\,\widehat{\mathcal E}(\Omega)\leq\mathfrak{c}_+$, since $\dim(^*\mathcal E(\Omega))=\card(^*\mathcal E(\Omega))=\mathfrak{c}_+$ (Example~\ref{Ex:  The Space *E(Omega)|C and its Dual} applied for $\kappa=\mathfrak{c}=\card\, \R$). 
		\item We observe that $J_{\mathcal D(\Omega)}$ is a differential $\mathcal E(\Omega)$-submodule of $^*\mathcal E(\Omega)$ and also a differential  $\C$-vector subspace of $^*\mathcal E(\Omega)$ with the property $^\sigma\!\mathcal E(\Omega)\cap J_{\mathcal D(\Omega)}=\{0\}$. Thus the mapping $f\mapsto{^*\!f}+J_{\mathcal D(\Omega)}$ is a $\mathcal E(\Omega)$-module embedding as well as a $\C$-vector space embedding of $\mathcal E(\Omega)$ into $\widehat{\mathcal E}(\Omega)$. Among other things, this embedding implies $\card\,\widehat{\mathcal E}(\Omega)\geq\mathfrak{c}$ and thus $\dim\widehat{\mathcal E}(\Omega)=\card\,\widehat{\mathcal E}(\Omega)$ due to the formula $\card\,\widehat{\mathcal E}(\Omega)=\max\{\dim\widehat{\mathcal E}(\Omega), \mathfrak{c}\}$ (Lemma~\ref{L: Span}).  The inequality $c\leq\dim\, \widehat{\mathcal E}(\Omega)\leq\mathfrak{c}_+$ implies that either $\dim\widehat{\mathcal E}(\Omega)=\mathfrak{c}$, or $ \dim\widehat{\mathcal E}(\Omega)=\mathfrak{c}_+$. In either case $\widehat{\mathcal E}(\Omega)$ can be embedded as a $\C$-vector subspace of $\mathcal D^*(\Omega)$ by Lemma~\ref{L: Subspace Lemma}, since $\dim\mathcal D^*(\Omega)=\mathfrak{c}_+$ (Example~\ref{Ex: D(Omega) and its Dual}).
	\end{itemize} 
\end{example}
\emph{Standardizations of non-standard results} are fascinating and often dramatic events in mathematics:
\begin{itemize}
	\item The most famous example of \emph{standardization} is certainly the creation of the modern calculus which \emph{standardizes} the old \emph{Leibniz-Newton-Euler Infinitesimal Calculus}. Getting rid of  \emph{infinitesimals} and replacing them with \emph{limits} can be described  as nothing less than a real \emph{drama or even revolution in mathematics} (Hall \& Todorov~\cite{HallTodLeibniz11}). The drama took several decades and it resulted - as side products - in the rigorous theory of real numbers used today,  set theory and mathematical logic. For those interested in the history of infinitesimals in the context of the Reformation, Counter-Reformation and English Civil War, we refer to the excellent book by Alexander~\cite{aAlexander2014}.
	\item A more recent example of a \emph{standardization of a non-standard result} gives us the Bernstein-Robinson~\cite{BernsteinRobinson66} solution of the invariant-subspace open problem for polynomially-compact operators on Hilbert space. Here the drama was not missing, either. Paul Halmos~\cite{prHalmos66} - in the role of Editor/Referee - decided to delay the acceptance of Bernstein-Robinson's manuscript for almost a year and meanwhile to translate (after \emph{sweating on} it, by his own words) the submitted manuscript  in the language of standard analysis, in an almost line-by-line correspondence to the original. Eventually, Halmos published his translation side by side with Bernstein-Robinson's article in the same issue of the Pacific Journal of Mathematics. Halmos's \emph{pseudo-standardization} of the Bernstein-Robinson solution (of an open until then problem in mathematics) considerably deflected the expected impact of one of the first and perhaps the most spectacular application of the newly-born non-standard analysis. Bitter feelings about this episode are still lingering in the non-standard community. The \emph{real standardization} of Bernstein-Robinson's solution came almost a decade later from
	Lomonosov~\cite{vLomonosov} in the framework of a more advanced and general operator theory. 
\end{itemize}

\subsection*{Acknowledgements}
The author thanks  Michael Oberguggenberger and Ivan Penkov for reading the earlier versions of the manuscript and for the insightful dialog. The author is also grateful to the anonymous referee for the numerous useful suggestions.

	%
	%
	%
	\bibliographystyle{jloganal}
	\bibliography{bibliography}

\end{document}